\documentclass[11pt]{amsart}

\usepackage{graphicx,amsfonts,amssymb,amsmath,amsthm,url,verbatim,amscd}
\usepackage{color}
\usepackage[usenames,dvipsnames]{xcolor}
\usepackage[normalem]{ulem}
\usepackage{pdfsync}
\usepackage{booktabs}
\usepackage[hyperfootnotes=false, colorlinks, citecolor=PineGreen, linkcolor=RoyalBlue, urlcolor=RoyalBlue]{hyperref}

\theoremstyle{plain}
\newtheorem*{theorem*}{Theorem}
\newtheorem{theorem}  {Theorem}[section]
\newtheorem{lemma}      [theorem]{Lemma}
\newtheorem{corollary}  [theorem]{Corollary}
\newtheorem{proposition}[theorem]{Proposition}

\newtheorem{prop}		[theorem]{Proposition}

\newtheorem{rem}		[theorem]{Remark}

\newtheorem{remark}		[theorem] {Remark}

\theoremstyle{definition}

\newtheorem{defn} [theorem]{Definition}

\newtheorem{example}    [theorem]{Example}

\allowdisplaybreaks
\newcommand{\leftexp}[2]{{\vphantom{#2}}^{#1}%
      \kern-\scriptspace%
      {#2}}

\renewcommand{\H}{\mathbb H}
\newcommand{\lcm}{\mathrm {lcm}}
\newcommand{\A}{{\mathbb A}}

\newcommand{\E}{{E}}
\renewcommand{\a}{{\mathfrak a}}
\newcommand{\f}{\operatorname{f}}

\newcommand{\Q}{{\mathbb Q}}
\newcommand{\Z}{{\mathbb Z}}

\newcommand{\R}{{\mathbb R}}
\newcommand{\C}{{\mathbb C}}
\newcommand{\bs}{\backslash}

\newcommand{\p}{\mathfrak p}

\newcommand{\OF}{{\mathfrak o}}
\newcommand{\GL}{{\rm GL}}

\newcommand{\SL}{{\rm SL}}
\newcommand{\SO}{{\rm SO}}

\newcommand{\sgn}{{\rm sgn}}
\newcommand{\St}{{\rm St}}

\newcommand{\Gal}{{\rm Gal}}

\renewcommand{\)}{{\upshape )}}

\newcommand{\mat}[4]{{\setlength{\arraycolsep}{0.5mm}\left[
      \begin{array}{cc}#1&#2\\#3&#4\end{array}\right]}}
\newcommand{\Aut}{{\rm Aut}}

\def\twist{{}}
\def\subscriptp{{}}

\def\vol{\operatorname{vol}}

\def\eps{\varepsilon}

\newcommand{\Cx}{\mathbb{C}^{\times}}
\newcommand{\Hb}{\mathbb{H}}
\newcommand{\Zp}{\mathbb{Z}_{p}}
\newcommand{\Zpx}{\mathbb{Z}_{p}^{\times}}
\newcommand{\Qx}{\mathbb{Q}^{\times}}

\newcommand{\Sc}{\mathcal{S}}

\newcommand{\Oix}{\mathfrak{o}^{\times}}
\newcommand{\Oi}{\mathfrak{o}}

\newcommand{\Xc}{\mathfrak{X}}

\newcommand{\Ga}{\mathfrak{G}}

\newcommand{\Fx}{F^{\times}}

\newcommand{\Wh}{\mathcal{W}}
\newcommand{\Vol}{\operatorname{Vol}}
\newcommand{\dxy}{d^{\times}y}

\newcommand{\dxv}{d^{\times}v}

\newcommand{\isom}{\cong}

\newcommand{\ifrm}{{\rm if}\,\,}
\newcommand{\andy}{\quad{\rm and}\quad}

\newcommand{\abs}[1]{\lvert{#1}\rvert}

\newcommand{\emb}{\hookrightarrow}

\begin{document}

\title{On the order of vanishing of newforms at cusps}

\author{Andrew Corbett}

\author{Abhishek Saha}

\address{Max-Planck-Institut f\"ur Mathematik, Vivatsgasse 7, 53111 Bonn, Deutschland}
\email{andrew.corbett@mpim-bonn.mpg.de}

\address{School of Mathematics, University of Bristol, Bristol, BS8 1TW, United Kingdom}
\email{abhishek.saha@gmail.com}

\begin{abstract}
Let $\E$ be an elliptic curve over $\Q$ of conductor $N$. We obtain an explicit formula, as a product of local terms, for the ramification index at each cusp of a modular parametrization of $E$ by $X_0(N)$. Our formula shows that the ramification index always divides 24, a fact that had been previously conjectured by Brunault as a result of  numerical computations. In fact, we prove a more general result which gives the order of vanishing at each cusp of a holomorphic newform of arbitary level, weight and character, provided that its field of rationality satisfies a certain condition.

The above result relies on a purely $p$-adic computation of possibly independent interest. Let $F$ be a non-archimedean local field of characteristic 0 and $\pi$ an irreducible, admissible, generic representation of $\GL_2(F)$. We introduce a new integral invariant, which we call the \emph{vanishing index} and denote $e_\pi(l)$, that measures the degree of ``extra vanishing" at matrices of level $l$ of the Whittaker function associated to the new-vector of $\pi$. Our main local result writes down the value of $e_\pi(l)$ in every case.
\end{abstract}

\maketitle

\section{Introduction}\label{s:introduction}

Let $\E$ be an elliptic curve over $\Q$ of conductor $N$ and let $\varphi: X_0(N) \rightarrow \E$ denote a modular parametrization defined over $\Q$ (the existence of $\varphi$ follows from the famous modularity theorem \cite{bcdt}). The points of $X_0(N)$ where the map $\varphi$ is ramified are of great interest; for instance they are relevant for the Birch and Swinnerton-Dyer conjecture \cite{mazswin}. We refer the reader to the papers \cite{brunault,delaunay} for further discussion, as well as for some numerical methods one can use to find these points.

In particular, it is natural to ask if $\varphi$ can ramify at a \emph{cusp} of $X_0(N)$. This problem was considered by Brunault \cite{brunault} who proved that if $\E$ is semistable ($N$ is squarefree), or more generally if the modular form $f$ attached to $\E$ has minimal conductor among its twists, then all critical points lie in the bulk, i.e., $\varphi$ is \emph{unramified }at all cusps of $X_0(N)$.

However, if $\E$ does not have the above properties, then $\varphi$ can ramify at certain cusps. Indeed, Brunault numerically computed the \emph{ramification index} $e_\varphi(\a)$ of $\varphi$ at each cusp $\a$ for all elliptic curves with $N\le 2000$  and found many examples where the experiments suggested that $e_\varphi(\a)$ is greater than 1. Any cusp of $X_0(N)$ can be represented by a rational number $\frac{a}{L}$ with $L | N$ and $(a, N)=1$; we refer to the integer $L$ as the \emph{denominator} of this cusp.\footnote{There are exactly $\phi(L, N/L)$ cusps of denominator $L$. The cusp at infinity is the unique one of denominator $N$.} As an immediate consequence of the fact that the Galois action on $X_0(N)$ is transitive on the set of cusps of a given denominator, it follows that  $e_\varphi(\a)$ depends only on the denominator of $\a$.

As a result of his computations, Brunault made the following experimental observations  for the ramification index $e_\varphi(\a)$ at a cusp $\a=\frac{a}{L}$ of denominator $L$.

\begin{enumerate}
\item  The integer $e_\varphi(\a)$ is always a divisor of 24.

\item If  $e_\varphi(\a)$ is even, then $v_2(L) \in \{2,3,4\}$ and $v_2(N)=2 v_2(L)$. (Throughout this paper, $v_p(a)$ denotes the highest power of $p$ dividing $a$.)

\item If    $e_\varphi(\a)$ is divisible by 8, then $v_2(L) =4$ and $v_2(N)=8$.

\item If    $e_\varphi(\a)$ is divisible by 3, then $v_3(L) =2$ and $v_3(N)=4$.
\end{enumerate}
In this paper we prove the following explicit formula for $e_\varphi(\a)$, which explains all the above observations.
 \begin{theorem}\label{t:main}Let $\E$, $N$, $\phi$ be as above. Let $\a$ be a cusp of $X_0(N)$ and let $L$ denote the denominator of $\a$. Then $e_\varphi(\a) = \prod_p p^{e_p}$ where the non-negative integers $e_p$ are given for each prime $p$ as follows.

 \begin{enumerate}
  \item If $p \ge 5$, then we have $e_p=0$.

\item If $p=3$, then we have $e_3 = 0$ except in the following special case:
\begin{enumerate}
 \item [(i)] If $v_3(N) =2v_3(L) = 4$ and the local component of $\E$ at 3 is a principal series representation, then we have $e_3 = 1$.

 \end{enumerate}

\item The case $p=2$.
 \begin{enumerate} \item [(ii)] If either $v_2(N) \le 2$ or $v_2(N) \neq 2 v_2(L) $, then we have $e_2=0.$

 \item [(iii)]  If $n_2:=v_2(N) = 2v_2(L) \ge 4$ and the local component\footnote{By the local component of $E$ at some prime $p$, we mean the local representation of $\GL_2(\Q_p)$ coming from the irreducible automorphic representation associated to $E$.} of $\E$ at the prime $2$ is a supercuspidal representation \(of conductor $2^{n_2}$\) whose minimal twist has conductor equal to $2^{n_2-1}$, then $e_2 =1$.

 \item [(iv)] If $v_2(N) =2v_2(L)=8$, and the local component of $\E$ at the prime $2$ is a principal series representation, then we have $e_2=3$.

\item [(v)] If we are not in any of the above three cases, then $e_2=2$.
 \end{enumerate}
 \end{enumerate}
\end{theorem}

Incidentally, Theorem \ref{t:main} also implies the result of Brunault described earlier. To see this, let the setup be as in the theorem and suppose that the modular form associated to $E$ has minimal conductor among its twists. If $p$ divides $N$, then the associated local representation of $\GL_2(\Q_p)$ cannot be a principal series representation (since a ramified principal series representation of $\GL_2(\Q_p)$ of trivial central character is not twist-minimal). So our Theorem implies that $e_3=0$. Furthermore, a result of Atkin-Li (see Theorem 4.4 of \cite{atkin-li}) implies that either $v_2(N) \le 2$ or $v_2(N)$ is odd. In either case, $e_2=0$ by our theorem. Hence $e_p = 0$ for all primes $p$ and therefore $e_\varphi(\a) = 1$ in this case.

Theorem \ref{t:main} is a special case of a more general result about modular forms that we describe now. Recall that the modularity theorem associates a cuspidal holomorphic newform $f$ (of weight 2, level $N$, trivial character, and rational Fourier coefficients) to our elliptic curve $\E$. The pullback by $\varphi$ of a N\'eron differential on $E$ is equal to a non-zero multiple of $\omega_f:=2 \pi i f(z) dz$, and for any cusp $\a$ of $X_0(N)$, the ramification index $e_\varphi(\a)$ equals $1 + \mathrm{ord}_{\a}(\omega_f)$. This latter quantity can be rewritten using the Fourier expansion of $f$ at $\a$. Let $L$ be the denominator of $\a$ and let  $w(\a) = N/(L^2, N)$. The integer $w(\a)$ is known as the \emph{width} of the cusp $\a$. Let $\sigma \in \SL_2(\Z)$ be such that $\sigma \a = \infty$. The Fourier expansion of $f$ at $\a$ looks as follows:
\begin{equation}\label{fourier1}
(f|_2\sigma^{-1})(z) = \sum_{n>0}a_f(n;\a)e^{\frac{2 \pi i nz}{w(\a)}},
\end{equation}
where the complex numbers $a_f(n;\a)$ are the Fourier coefficients of $f$ at the cusp $\a$. Strictly speaking, the Fourier coefficients $a_f(n;\a)$ depend not just on $\a$ but also on the choice of $\sigma$.
However, if $a'_f(n;\a)$ denotes the coefficient
obtained by a different choice $\sigma'$,
then one has $a_f(n;\a) = e^{\frac{2 \pi ib n}{w(\a)}} a'_f(n;\a)$
for some integer $b$.

As $e_\varphi(\a)$ equals $1 + \mathrm{ord}_{\a}(\omega_f)$, it follows that
\begin{equation}\label{e:ramindexfourier}
e_\varphi(\a) = \min\{n>0: a_f(n;\a) \neq 0 \}.
\end{equation}
In other words, computing the ramification index at a cusp reduces to finding the order of vanishing of the corresponding newform $f$ at that cusp. It is natural to try to solve this problem for all newforms $f$ (or arbitrary weight and character) and not just those coming from elliptic curves.

So, let $f$ be a holomorphic cuspidal newform of weight $k$, level $N$, and character $\chi$, that is, $$f|_k \mat{a}{b}{c}{d} = \chi(d) f$$ for all $\mat{a}{b}{c}{d} \in \Gamma_0(N)$.  Let $M$ denote the conductor of $\chi$ (so  $M$ divides $N$). As before, let $\a = \frac{a}{L}$ be a cusp of $X_0(N)$,  let $\sigma \in \SL_2(\Z)$ be such that $\sigma \a = \infty$, and let $w(\a) = N/(L^2, N)$. The presence of the character $\chi$ complicates the Fourier expansion slightly.  Indeed, an easy calculation shows that for any integer $t$, we have, $(f|_k\sigma^{-1})(z+tw(\a)) = \chi(1 + atw(\a)L) (f|_k\sigma^{-1})(z).$ Therefore, to ensure that the value of the character is 1, we need $M| Lw(\a)t$, or equivalently, $\frac{M}{(Lw(\a), M)} |t$. So define\footnote{
Note that $\delta(\a)$ equals $w(\a)$ for all cusps $\a$ if and only if the conductor $M$ of $\chi$ divides $N_1$ where $N_1$ is the smallest integer such that $N|N_1^2$.} $\delta(\a):= w(\a) \frac{M}{(Lw(\a), M)} = \frac{[L^2, N, LM]}{L^2}$. We have $(f|_k\sigma^{-1})(z+\delta(\a)) =  (f|_k\sigma^{-1})(z),$   and so the Fourier expansion of $f$ at $\a$ is as follows:
\begin{equation}\label{fourier2}
(f|_k\sigma^{-1})(z) = \sum_{n>0}a_f(n;\a)e^{\frac{2 \pi i nz}{\delta(\a)}}.
\end{equation}
If $\a$ is the cusp at infinity (i.e., $L=N$), then we simply use $a_f(n)$ to denote $a_f(n;\a)$; these are the usual Fourier coefficients of $f$. Define the quantities
\begin{equation}
e_f(\a) :=  \min\{n>0: a_f(n;\a) \neq 0 \}, \quad e_f(L) = \min_{\operatorname{denominator}(\a)= L} e_f(\a).
\end{equation}
Our main global result gives an explicit formula for $e_f(L)$ as a product of local terms which depend on the representations $\pi_p$.

\begin{theorem}\label{t:main2}

Let $f$ be a cuspidal holomorphic newform of weight $k$, level $N= \prod_p p^{n_p}$ and character $\chi$, and let $\pi \simeq \otimes_p \pi_p$ be the automorphic representation associated to $f$. Then for any integer $L= \prod_p p^{l_p}$ dividing $N$, we have $e_f(L) = \prod_p p^{e_{\pi_p}(l_p)}$ where for all irreducible admissible generic representations $\pi_p$ of $\GL_2(\Q_p)$, and all $0 \le l_p \le n_p$, the ``vanishing index" $e_{\pi_p}(l_p)$ is given as follows:

\begin{enumerate}
\item If $p\ge 5$, then we have $e_{\pi_p}(l_p)=0$.

\item If $p=3$, then $e_{\pi_p}(l_p)=0$, except in one case:

\begin{itemize}
\item $e_{\pi_p}(l_p)=1$ if
\begin{enumerate}
\item[(i)] $\pi_p=\chi_{1}\boxplus\chi_{2}$ with   $a(\chi_{1})=a(\chi_{2})=l_p$, $n_p=2l_p\geq 4$, and $a(\chi_{1}\chi_{2}^{-1})=l_p$.
\end{enumerate}
\end{itemize}

\item If $p=2$, then $e_{\pi_p}(l_p)=0$, aside from the following exceptions:

\begin{itemize}

\item $e_{\pi_p}(l_p)=1$ when

\begin{enumerate}
\item[(ii)] $\pi_p=\chi_{1}\boxplus\chi_{2}$ with $a(\chi_{1})$ and $a(\chi_{2})$ both at least 2, $a(\chi_{1}) \neq a(\chi_2)$, and either $l_p=a(\chi_{1})$ or $l_p=a(\chi_{2})$;
\item[(iii)] $\pi_p=\chi\pi_{0}$ and $n_p=2l_p\geq 4$ where $\pi_{0}$ is a supercuspidal representation with $a(\pi_{0})=n_p-1$ and $\chi$ is a  character such that $a(\chi)=n_p/2$.
\end{enumerate}

\item $e_{\pi_p}(l_p)=2$ when

\begin{itemize}
\item[(iv)] $\pi_p=\chi\St$ with $a(\chi)\ge 2$ and $n_p=2l_p = 2 a(\chi)\geq 4$;
\item[(v)] $\pi_p=\chi_{1}\boxplus\chi_{2}$ with $n_p=2l_p \ge 4$, $a(\chi_{1})=a(\chi_{2})=l_p$, $\chi_1 \chi_2^{-1} \notin \{| \cdot |, \ | \cdot |^{-1} \}$, and  $a(\chi_{1}\chi_{2}^{-1})<l_p-1$;
\item[(vi)] $\pi_p=\chi\pi_{0}$ and $n_p=2l_p\geq 4$ where  $a(\chi)=n_p/2$ and $\pi_{0}$ is a minimal supercuspidal representation (see Definition \ref{defminimal}) with $a(\pi_{0})\leq n_p-2$.
\end{itemize}

\item $e_{\pi_p}(l_p)=3$ when

\begin{itemize}
\item[(vii)] $\pi_p=\chi_{1}\boxplus\chi_{2}$ with $n_p=2l_p\ge 6$, $a(\chi_{1})=a(\chi_{2})=l_p$ and $a(\chi_{1}\chi_{2}^{-1})=l_p-1$.

\end{itemize}

\end{itemize}
\end{enumerate}
\end{theorem}

For unfamiliar notation and a general definition of the vanishing index for representations of $\GL_2$ over arbitrary non-archimedean local fields, as well as a formula in every case, we refer the reader to \S\ref{s:local} and in particular Theorem \ref{thm_andys_theorem}. Note that Theorem \ref{t:main2} implies that $e_f(L)$ equals 1 unless at least one of the integers 16 or 81 divides $N$. In order to read off $e_f(L)$ in these cases using Theorem \ref{t:main2}, it is necessary to  know only the local representations of $\GL_2(\Q_2)$ and $\GL_2(\Q_3)$  associated to $f$; this can be achieved using the algorithm presented in \cite{loeffler-weinstein}.

The proof of Theorem \ref{t:main2}  follows from a local computation, which is the heart of this paper. From the adelic viewpoint, the Fourier coefficient $a_f(n;\a)$ is equal to the value of the global Whittaker newform associated to $f$ at a certain adelic matrix. By the uniqueness of the Whittaker model, the global Whittaker newform factors as a product of local newforms. So we are reduced to solving the problem of ``extra vanishing" at matrices of level $l$ of the local Whittaker newform associated to an arbitary irreducible, admissible, generic representation of $\GL_2(\Q_p)$.  We do this in the more general context of an arbitrary non-archimedean local field of characteristic zero (this also means that Theorem \ref{t:main2} can be generalized in a straightforward manner to automorphic forms of $\GL_2$ over number fields, if one so chooses).  Our key tool is a certain ``basic identity" that was proved by the second author in \cite{saha-large-values} (this identity is obtained from the Jacquet--Langlands local functional equation via some elementary Fourier analysis over a finite abelian group). Ultimately, we are reduced to the problem of counting characters whose twists have certain prescribed conductors; this is done in  \S\ref{sec_proof_of_thm}.

Theorem \ref{t:main2} gives an exact formula for the \emph{minimum value} of $e_f(\a)$ taken over cusps $\a$ of a fixed denominator. Ideally we would like a formula for $e_f(\a)$ for each cusp $\a$. In general, such a refined result cannot be deduced from Theorem \ref{t:main2}; however, we now give a key case where this is possible.
\begin{proposition}\label{prop:rational}
Let $f$ be as in Theorem \ref{t:main2}, and let $\Q(f)$ be the (number) field\footnote{In fact, it is known that $\Q(f)$ is always a subfield of a CM field. Moreover, by strong multiplicity, it follows that for every positive integer $t$, $\Q(f)$ is generated by the quantities $a_f(n)$ where $n$ ranges over only the positive integers that are coprime to $t$.} generated by all the Fourier coefficients $a_f(n)$. Suppose for some divisor $L$ of $N$ we have $\Q(f) \cap \Q(e^{\frac{2 \pi i}{(L, N/L)}}) = \Q$ (where we think of all our number fields as subsets of the complex numbers). Then all cusps $\a$ of denominator $L$ have the same value of $e_f(\a)$; in other words, $e_f(\a) = e_f(L)$.
\end{proposition}
If the modular form $f$ comes from an elliptic curve, then we have $\Q(f) = \Q$; therefore the condition $\Q(f) \cap \Q(e^{\frac{2 \pi i}{(L, N/L)}}) = \Q$ is trivially satisfied. In this case, Theorem \ref{t:main2} gives an exact formula for $e_f(\a) = e_\varphi(\a)$ which is precisely what is stated in Theorem \ref{t:main}. Several of the cases of $e_{\pi_p}(l_p)>0$ that are described in Theorem \ref{t:main2} do not appear in Theorem \ref{t:main}. This is a reflection of the fact that the modular forms $f$ associated to elliptic curves are in some sense special.

Proposition \ref{prop:rational} is a reflection of the fact that there exists an action of $\Aut(\C)$ on the set of cohomological automorphic representations of $\GL_2$, which factors into a product of local actions and is compatible with the classical action of $\Aut(\C)$ on Fourier coefficients.  The local action can be studied via local Whittaker newforms, and the condition in Proposition \ref{prop:rational} ensures that this action is transitive on the cusps of a given denominator $L$. For the details, see \S\ref{s:localrational} and \S\ref{s:globalrational}.
In fact, Proposition \ref{prop:rational} remains true when $\Q(f)$ is replaced by the the compositum of the fields of rationality of the local representations $\pi_p$  over the (finitely many) primes $p$ with $p^2|N$ (see Proposition \ref{prop:globalrational}). When $k=2$, Proposition \ref{prop:rational}  also follows from Lemma 1.3 of \cite{brunault}; this alternate method, however, does not give the stronger result described in the previous sentence.

We end this introduction with a few further remarks about the condition in Proposition \ref{prop:rational}. Let $N_0$ denote the largest integer whose square divides $N$. If it is true that
\begin{equation}\label{e:rationality}
\Q(f) \cap \Q(e^{\frac{2 \pi i}{N_0}}) = \Q,
\end{equation}
then the condition in Proposition \ref{prop:rational} is satisfied for all $L|N$. So in this case, $e_f(\a) = e_f(L)$ for all cusps $\a$ of denominator $L$, and therefore Theorem \ref{t:main2} gives an exact formula for $e_f(\a)$. While we are unaware of any results describing how often a form $f$ satisfies the rationality condition \eqref{e:rationality}, a perusal of the LMFDB database makes it clear that this condition is indeed satisfied the vast majority of the time\footnote{However, we believe that this condition is never satisfied when a high power of an odd prime divides $N$.} for $1 \le N \le 100$, $2\le k \le 12$. An interesting low weight case where this condition is not satisfied occurs when $k=2, M=1$, and $N=567 = 3^4 \times 7.$ For this data, there exists a form $f$ such that $\Q(f)$ is the maximal totally real subfield of $\Q(e^{2 \pi i/9})$, and for which numerical experiments performed by Brunault strongly indicate that $e_f(1/9) = e_f(2/9) = 3$ but $e_f(4/9) = 6$.\footnote{This example and a few others were discovered by Fran\c{c}ois Brunault [personal communication, July 2012].} The above example shows that the condition in Proposition \ref{prop:rational} is indeed necessary. Another example, with $N=625$, is given in Remark 5.1 of \cite{brunault}.  More generally, it was shown by Fran\c{c}ois Brunault and Paul Nelson (personal communication, July 2012) that if $p\ge 5$ is a prime and $f$ is a newform with $M=1$, $N=p^4$, such that the local component $\pi_p$ is a principal series representation, then we have that $e_f(\frac{a}{p^2}) =1$ for only about half the values of $a$, and $e_f(\frac{a}{p^2}) >1$ for the remaining half! This follows from the automatic vanishing of certain exponential sums modulo $p^2$. For any $f$ as above, the corresponding local field of rationality $\Q(\pi_p)$ (which is contained in $\Q(f)$) intersects non-trivially the cyclotomic field $\Q(e^{2 \pi i/p})$.

The above examples make it clear that when the rationality condition \eqref{e:rationality} is not satisfied, the problem of computing $e_f(\a)$ for individual cusps $\a$  is a subtle one. In fact, one can show that this problem is equivalent to understanding the vanishing of certain $p$-adic analogues of hypergeometric functions. Further investigation of these functions from an analytic point of view will be done in forthcoming work of the second author with Yueke Hu.

\subsection*{Notations}

We collect here some general notations that will be used throughout this paper. Additional notations will be defined where they first appear in the paper.

Given two integers $a$ and $b$, we use $a|b$ to denote that $a$ divides $b$, and we use $a|b^\infty$ to denote that $a|b^n$ for some positive integer $n$.  We let $v_a(b)$ denote the largest non-negative integer such that $a^{v_a(b)} | b$. We will occasionally use the shorthand notation $e(x) = e^{2 \pi i x}$.

The group $\Gamma_0(N)$ consists of those matrices $\left(\begin{smallmatrix}a&b\\c&d
\end{smallmatrix}\right) \in \SL_2(\Z)$ such that $N\mid c$. The subgroup $\Gamma_1(N)$ consists of those matrices in $\Gamma_0(N)$ with the added property that $a \equiv d \equiv 1 \mod{N}$.
Let $\H$ denote the upper half plane and $\GL_2(\R)^+$ the group of real two-by-two matrices with positive determinant. For $z \in \H$, $\left(\begin{smallmatrix}a&b\\c&d
\end{smallmatrix}\right) \in \GL_2(\R)^+$, we let $\left(\begin{smallmatrix}a&b\\c&d
\end{smallmatrix}\right)z=\frac{az+b}{cz+d} \in \H$ be the point obtained by M\"obius transformation. Let $X_0(N)$ denote the usual modular curve obtained by the compactification of $\Gamma_0(N)\bs \H$. Given a function $f$ on $\H$, an integer $k$, and some $\gamma = \left(\begin{smallmatrix}a&b\\c&d
\end{smallmatrix}\right) \in \GL_2(\R)^+$, we define a function $f |_k \gamma$ on $\H$ via $(f |_k \gamma)(z) = \det(\gamma)^{k/2} (cz+d)^{-k} f(\gamma z).$

We shall always assume every character  is continuous (but not necessarily unitary). For a complex representation $\pi$ of some group $H$ and an automorphism $\sigma$  of $\C$, there is a complex representation ${}^\sigma \pi$ of $H$ defined as follows. Let $V$ be the space of $\pi$ and let $V'$ be any vector space such that $t: V \rightarrow V'$ is a $\sigma$-linear isomorphism (that is, $t(v_1 +v_2) = t(v_1) + t(v_2)$ and $t(\lambda v) = \sigma(\lambda) t(v)$). We define the representation $( \leftexp{\sigma}{\pi}, V')$ via $\leftexp{\sigma}{\pi}(g) = t \circ \pi(g) \circ t^{-1}$. It can be shown easily that the representation $\leftexp{\sigma}{\pi}$ does not depend on the choice of $V'$ or $t$. We define $\Q(\pi)$ to be the fixed field of the set of all automorphisms $\sigma$ such that $\leftexp{\sigma}{\pi} \simeq \pi$.

\section{Local computations}\label{s:local}

\subsection{Notations and background}\label{sec_notation}

\subsubsection{Notations for local fields}

Let $F$ be a non-archimedean local field of characteristic zero and let $G=\GL_{2}(F)$. We denote by $\Oi$ the ring of integers of $F$ and denote by $\p$ the maximal ideal of $\OF$. We fix a uniformiser, that is a generator of $\p$, and denote it by $\varpi$; we let $q=\# (\Oi/\p)$. Let $\abs{\,\cdot\,}$ be the absolute value on $F$, normalised so that $\abs{\varpi}=q^{-1}$, and
$v$ the valuation on $F$ defined via $|x| = q^{-v(x)}$. The subgroups $U_k$ of $\OF^\times$ are defined as follows: $U_{k}=1+\varpi^{k}\Oi$ for $k>0$, and $U_{0}=\Oix$. Let $dy$ be the Haar measure on $F$, normalised so that $\Vol(\Oi,dy)=1$, and $\dxy$ the Haar measure on $\Fx$, normalised so that $\Vol(\Oix,\dxy)=1$. Finally, $\zeta(s)=(1-q^{-s})^{-1}$ denotes the (local) zeta-function of $F$.

\subsubsection{Characters of $\Oix$}

For a character $\chi\colon \Fx\rightarrow\Cx$ we denote by $a(\chi)$ (the exponent of) its \textit{conductor}; this is the smallest integer $k\geq 0$ such that $\chi(U_k)=\{1\}$. We say $\chi$ is \textit{unramified} if $a(\chi)=0$. Let
\begin{equation*}
\Xc=\left\lbrace\,\mu\colon\Fx\rightarrow\Cx \,:\, \mu(\varpi)=1\, \right\rbrace
\end{equation*}
so that $\Xc$ is isomorphic to the group of continuous characters on $\Oix$. Any character of $\Xc$ is unitary and of finite order. We also consider characters in $\Xc$ of particular conductors, duly introducing the notation:
\begin{equation*}
\Xc_{k}=\left\lbrace\,\mu\in\Xc\,:\,a(\mu)\leq k\,\right\rbrace\andy\Xc_{k}'=\left\lbrace\,\mu\in\Xc\,:\,a(\mu)= k\,\right\rbrace.
\end{equation*}
Note that $\{1\}=\Xc_{0}\subset \Xc_{1}\subset\cdots \subset\Xc_{k}\subset\cdots \subset\Xc$ as subgroups. We have for each $k \ge 1$, $\#\Xc_{k}=q^{k-1}(q-1)$, $\#\Xc_{1}' =  q-2$, and for $k\geq 2$, $\#\Xc_{k}' =q^{k-2}(q-1)^{2}$. Furthermore, for all $l \ge k \ge l/2 \ge 1$, we have $\Xc_l / \Xc_k \cong U_k/U_l \cong \OF/\p^{l-k}.$

We now answer a  question which will frequently sprout up in our computations.

\begin{lemma}\label{lem_monkey}
Let $k\geq 2$ and let $\chi$ be a character of $\Fx$ with conductor $a(\chi)=k$.

\begin{enumerate}

\item Then there exists a character $\mu\in\Xc$ such that $a(\mu)=a(\mu\chi)=k$ if and only if $q>2$.

\item   Let $q=2$.  If $k>2$, then there exists a character $\mu\in\Xc$ such that $a(\mu)=k$ and $a(\mu\chi)=k-1$. If $k=2$, then for any $\mu\in\Xc$ satisfying $a(\mu)=2$, we have $a(\mu\chi)=0$.
\end{enumerate}

\end{lemma}

\begin{proof}
For the first assertion, asking for such a $\mu$ is equivalent to demanding $\mu\not\equiv 1$ and $\mu\chi\not\equiv 1$ in the group $\Xc_{k}/\Xc_{k-1}$ (where we abuse notation by writing $\chi$ for the restriction $\chi\vert_{\Oix}$). Since the group $\Xc_{k}/\Xc_{k-1}$ has order $q$, and we need to avoid $\mu\equiv 1$ and $\mu\equiv \chi^{-1}$, there always exists such a $\mu$ whenever $q-2>0$.

Now let $q=2$. If $k>2$, then, as $\#(\Xc_{k}/\Xc_{k-1})=2$ but $\#(\Xc_{k}/\Xc_{k-2})=4$, we can always find a $\mu$ with $a(\mu)=k$ and $a(\mu\chi) = k-1$. On the other hand, if $k=2$, then $\# \Xc_{2}= 2$, the class of these characters being those of $1$ and $\chi$, and the class of $\chi^{2}$ is therefore trivial.
\end{proof}

We shall also require an extension of Lemma \ref{lem_monkey} for two characters $\chi_1$, $\chi_2$ of the same conductor.

\begin{lemma}\label{lem_monkeys}
Let $k\geq 2$ and let $\chi_1, \chi_2$ be characters of $\Fx$ such that $a(\chi_1)=a(\chi_2)=k$.
\begin{enumerate}

\item Suppose that either a) $q>3$ or b) $q=3$ and $a(\chi_1 \chi_2^{-1}) <k$. Then there exists a character $\mu\in\Xc$ such that $a(\mu)=a(\mu\chi_1)=a(\mu\chi_2)=k$.
\item Suppose that $q=3$ and $a(\chi_1 \chi_2^{-1}) =k$. Then $\max\{a(\mu\chi_1)+ a(\mu\chi_2): \mu\in\Xc'_k\} = 2k-1$.
\item  Suppose that $q=2$ and  $a(\chi_1 \chi_2^{-1}) <k-1$.  If $k>2$, then there exists a character $\mu\in\Xc$ such that $a(\mu)=k$ and $a(\mu\chi_1)=a(\mu\chi_2) =k-1$. If $k=2$, then for any $\mu\in\Xc$ satisfying $a(\mu)=2$, we have $a(\mu\chi_1)=a(\mu\chi_2)=0$.
\item Suppose that $q=2$ and  $a(\chi_1 \chi_2^{-1}) =k-1$.  Then automatically $k\ge 3.$ If $k \ge 4$, then $\max\{a(\mu\chi_1)+ a(\mu\chi_2): \mu\in\Xc'_k\} = 2k-3$. If $k=3$, then for any $\mu\in\Xc$ satisfying $a(\mu)=3$, we have  $\{a(\mu\chi_1), a(\mu\chi_2)\} = \{2,0\}$.

\end{enumerate}
\end{lemma}

\begin{proof}

This follows the same routine as the proof of Lemma \ref{lem_monkey}. In case (1) we use a) $\#(\Xc_{k}/\Xc_{k-1})=q>3$, so that there is at least a fourth class inequivalent to $1$, $\chi_{1}^{-1}$, or $\chi_{2}^{-1}$; in b) we allow $q=3$ but force $\chi_{1}\equiv \chi_{2} \mod \Xc_{k-1}$ so $\mu$ should be chosen outside just two classes and $\#(\Xc_{k}/\Xc_{k-1})>2$.

In case (2) we have $q=3$ and $\chi_{1}\not\equiv\chi_{2} \mod \Xc_{k-1}$. Necessarily, $\mu$ is equivalent to precisely one of $\chi_{1}^{-1}$ or $\chi_{2}^{-1}$ in $\#(\Xc_{k}/\Xc_{k-1})$. But there are always at least five non-trivial classes (in the worst case $q=3$ and $k=2$) in $\#(\Xc_{k}/\Xc_{k-2})$ so we can find $\mu$ such that $\{a(\mu\chi_1), a(\mu\chi_2)\} = \{k, k-1\}$. In cases (3) and (4) we apply this reasoning, mutatis mutandis, for $q=2$.
\end{proof}

\subsubsection{The Gauss sum}

 We fix once and for all an additive character $\psi\colon F/\Oi\rightarrow \Cx$  on $F$ such that $\psi$ is trivial on $\OF$ but non-trivial on $\p^{-1}$. For each $a\in\Fx$ and $\mu\in\Xc$ define the \textit{Gauss sum}:
\begin{equation*}
\Ga(a,\mu)=\int_{\Oix}\psi(ay)\,\mu(y)\,\dxy.
\end{equation*}
\begin{lemma}\label{lem_pine}
Let $v\in\Oix$, $r\in\Z$ and $\mu\in\Xc$.
Then
\begin{equation}\label{eq1lem}
\Ga(v\varpi^{-r},1)= \begin{cases}1 & \text{ if }  r \le 0, \\ -\zeta(1)q^{-1} &\text{ if }  r =1, \\ 0 & \text{otherwise.} \end{cases}
\end{equation}

 If $a(\mu)>0$, then
 \begin{equation}\label{e:eq2}
 \Ga(v\varpi^{-r},\mu)=
 \begin{cases}
 \zeta(1)q^{-r/2}\eps(1/2, \mu^{-1}, \psi)\mu^{-1}(v)  &\text{ if } r=a(\mu), \\ 0  &\text{ otherwise. }
\end{cases}
\end{equation}
Above, $\eps(1/2, \mu^{-1}, \psi)$ is the usual $\GL(1)$-epsilon factor \(or root number\) associated to the character $\mu^{-1}$ and the additive character $\psi$; in particular $|\eps(1/2, \mu^{-1}, \psi)|=1$.

\end{lemma}

\begin{proof}
We first prove \eqref{eq1lem}. By our normalisation of measures we have $$\Ga(v\varpi^{-r},1)= \zeta(1)\left(\int_{\OF} \psi(v \varpi^{-r}y) dy -  q^{-1}\int_{\OF} \psi(v \varpi^{-r+1}y) dy \right).$$ The result \eqref{eq1lem} now follows immediately from the orthogonality of additive characters and the fact that $\psi$ is trivial on $\OF$ but not on $\p^{-1}$. Next we prove \eqref{e:eq2}. The vanishing of the Gauss sum when $a(\mu) \neq r$ follows from \cite[Lemma 7-4]{ramval}. On the other hand, when $a(\mu)=r$ we get using $U_r$-invariance $$\Ga(v\varpi^{-r},\mu)= \mu^{-1}(v)\vol(U_r) \sum_{U_r \bs \OF^\times} \psi(\varpi^{-r} y) \mu(y).$$ The result now follows from \cite[(7.6)]{ramval} (where $\eps(1/2, \mu^{-1}, \psi)$ is denoted $W(\mu^{-1})$).
\end{proof}

\subsubsection{Definitions of matrix groups}

Let $K=\GL_{2}(\Oi)$ and for an integer $n\geq 0$ let

\begin{equation}\label{e:defk1}
K_{1}(n)=\left\lbrace\,\begin{pmatrix} a&b\\c&d\end{pmatrix}\in K\ :\ c\in \varpi^{n}\Oi,\ a\in U_n\,\right\rbrace, \quad K'_{1}(n)=\left\lbrace\,\begin{pmatrix} a&b\\c&d\end{pmatrix}\in K\ :\ c\in \varpi^{n}\Oi,\ d\in U_n\,\right\rbrace.
\end{equation}
Define the matrices
$
a(y)=\begin{pmatrix}
y&\\&1
\end{pmatrix},$ $
w=\begin{pmatrix}
&1\\-1&
\end{pmatrix},$ $
n(x)=\begin{pmatrix}
1&x\\&1
\end{pmatrix}
$
for each $y\in\Fx$ and $x\in F$. Define subgroups
$$N =
\{n(x):  x\in F \}, \quad
A = \{a(y): y\in F^\times \}, \quad
Z =\{ z(t):
t \in F^\times \}$$ of $G$. For any $t,l\in\Z$ and $v\in\Oix$ define
\begin{equation*}
g_{t,l,v}=a(\varpi^{t})wn(v\varpi^{-l})= \begin{pmatrix}
&\varpi^{t}\\
-1&-v\varpi^{-l}
\end{pmatrix}.
\end{equation*}
\begin{remark} Suppose that $n$ is fixed. Then, for each $g \in G$ there is a unique integer $l$ satisfying $0 \le l \le n$ such that $g \in ZNg_{t,l,v}K_1(n)$ for some $t \in \Z$, $v \in \OF^\times$ \(see Lemma 2.13 of \cite{saha-large-values}\).
\end{remark}

\subsubsection{Notation for representations of $G$}

For an irreducible, admissible, generic representation $\pi$ of $G$ we define (the exponent of) its \textit{conductor} $a(\pi)$ to be the smallest integer $n\geq 0$ such that the space of $K_{1}(n)$-fixed vectors in $\pi$ contains a non-zero vector. It is well known that the space of $K_{1}(a(\pi))$-fixed vectors is one-dimensional. If $a(\pi)=0$, then $\pi$ is said to be \textit{unramified}.

For a character $\chi$ of $\Fx$ we write the character twist of $\pi$ as $\chi\pi$ which is defined to be the representation of $G$ given by $g\mapsto \chi(\det(g))\pi(g)$. The central character of $\pi$ shall be denoted $\omega_{\pi}$ and the representation contragredient to $\pi$ is denoted $\tilde{\pi}$; as $G=\GL_{2}(F)$ we can realise $\tilde{\pi}=\omega_{\pi}^{-1}\pi$.

For two characters $\chi_{1},\chi_{2}$ of $\Fx$, let $\chi_{1}\boxplus \chi_{2}$ denote the (normalised) principal series representation of $G$ parabolically induced from the character $\chi_{1}\otimes\chi_{2}$ on the standard Levi subgroup $B=ZNA$ of $G$. The parabolic induction is normalised by multiplying $\chi_{1}\otimes\chi_{2}$ by $\abs{y}^{1/2}$ (which is the square-root of the modulus character on $B$) before inducing; see equation (4.9) of \cite{gelbart}. This ensures that $\chi_{1}\boxplus \chi_{2}$ is unitary whenever $\chi_1$ and $\chi_2$ are unitary. The representation $\chi_{1}\boxplus \chi_{2}$ is irreducible if and only if $\chi_1 \chi_2^{-1} \notin \{| \cdot |, \ | \cdot |^{-1}  \}$. This condition is automatically satisfied if $a(\chi_1) \ne a(\chi_2)$, or more generally, if $a(\chi_1 \chi_2^{-1}) \neq 0$.  Whenever $\chi_{1}\boxplus \chi_{2}$ is irreducible, it is true that $\chi_{1}\boxplus \chi_{2}\isom \chi_{2}\boxplus \chi_{1}$ and $\omega_{\pi}=\chi_{1}\chi_{2}$.

For an irreducible, admissible, generic representation $\pi$ of $G$, we let $L(s, \pi)$ denote the local $L$-factor and $\eps(s, \pi)=\eps(s, \pi, \psi)$ denote the local $\eps$-factor with respect to our fixed additive character $\psi$; these factors are defined by their existence in \cite[Theorem 2.18]{jacquet-langlands}. It is known that $\eps(s, \pi, \psi)$ is a non-zero complex number and $\eps(s, \pi, \psi)= \eps(1/2, \pi | \cdot |^{1/2 -s}, \psi)= q^{(1/2-s)a(\pi)}\eps(1/2, \pi, \psi)$.

\subsubsection{A classification of representations of $G$}\label{sec_classification}

For our analysis, we give a classification of the irreducible, admissible, generic representations $\pi$ of $G$ satisfying $a(\pi)\geq 1$. This classification is well known in the literature (for example, see \cite[Theorem 4.18, Theorem 4.21, and Remark 4.25]{gelbart}).

\begin{enumerate}\vspace{0.05in}

\item $\pi \simeq\chi\St$, a twist of the Steinberg representation $\St$ by an unramified  character $\chi$; these have $a(\pi)=1$, $\omega_{\pi}=\chi^{2}$ and $L(s,\pi)=L(s,\chi\abs{\,\cdot\,}^{1/2})$.\vspace{0.05in}

\item $\pi \simeq \chi_{1}\boxplus \chi_{2}$, where $\chi_{1},\chi_{2}$ are  characters of $\Fx$ with $a(\chi_{1})>0=a(\chi_{2})$. These have $a(\pi)=a(\chi_{1})\geq 1$, $\omega_{\pi}=\chi_{1}\chi_{2}$ and $L(s,\pi)=L(s,\chi_{2})$.\vspace{0.05in}

\item $\pi$ satisfies $L(s,\pi)=1$. In this case we enumerate the following subcases:\vspace{0.05in}
\begin{enumerate}
\item $\pi \simeq\chi\St$, where $a(\chi)>0$; these have $a(\pi)=2a(\chi)\geq 2$.\vspace{0.05in}
\item $\pi \simeq\chi_{1}\boxplus \chi_{2}$, where $\chi_{1},\chi_{2}$ are characters with $a(\chi_{1})\geq a(\chi_{2})>0$ and $\chi_1 \chi_2^{-1} \notin \{| \cdot |, \ | \cdot |^{-1} \}$; these have $a(\pi)=a(\chi_{1})+a(\chi_{2})\geq 2$.\vspace{0.05in}
\item $\pi$ is supercuspidal; these also have $a(\pi)\geq 2$.
\end{enumerate}
\end{enumerate}

By a well-known result of Tunnell \cite[Prop. 3.4]{tunnell} it follows that a supercuspidal representation $\pi$ always satisfies $a(\omega_\pi) \le a(\pi)/2$.

\subsubsection{Conductors of character-twists of representations}\label{sec_character_twists}

\begin{defn}\label{defminimal}
We call $\pi$ \textit{minimal} if $a(\pi)=\min\{ a(\chi\pi)\,:\,\chi\in\Xc\}$.
\end{defn}

\begin{example}\label{exx}

Representations of type (1) and type (2) are always minimal. Representations of types (3.a) and (3.b) are never minimal.

\end{example}

\begin{lemma}\label{lem_wal}

Let $\pi$ be an irreducible, admissible, generic representation of $G$ and let $\chi$ be a  character of $\Fx$. We have that the conductor
\begin{equation}\label{eq_pecan}
a(\chi\pi)\leq \max\left\lbrace\, a(\pi),\,a(\omega_{\pi})+a(\chi),\,2a(\chi)\, \right\rbrace.
\end{equation}
Moreover, we have equality in \eqref{eq_pecan} in each of the following cases: (i) $a(\omega_\pi) \le a(\pi)/2$, $\pi$ minimal, (ii) $a(\omega_\pi) \le a(\pi)/2$, $a(\pi) \neq 2 a(\chi)$, (iii) $a(\omega_\pi) > a(\pi)/2$, $\pi$ minimal, $a(\pi) \neq  a(\chi)$, (iv) $a(\omega_\pi) > a(\pi)/2$, $a(\chi) \notin \{a(\omega_\pi), a(\pi) - a(\omega_\pi)\}$.

\end{lemma}

\begin{proof}
In the supercuspidal case, this follows from \cite[Proposition 3.4]{tunnell}; in fact the proof there holds verbatim for all square-integrable representations $\pi$. The Lemma also holds trivially when $a(\pi)=0$, as in this case $a(\chi \pi) = 2 a(\chi)$. So  we are left to only prove the lemma for the ramified principal series representations $\chi_1 \boxplus \chi_2$, types (2) and (3.b) in our notation. Of these, the former is always minimal and the latter is never minimal. Moreover in those cases, we have the formulas: $a(\pi)=a(\chi_{1})+a(\chi_{2})$, $a(\omega_\pi) = a(\chi_1 \chi_2)$, $a(\chi \pi)=a(\chi \chi_{1})+a(\chi \chi_{2})$. The problem is thus reduced to computing conductors of one-dimensional representations; each case may thus be determined via the observation $a(\mu_1\mu_2)\leq \max\{a(\mu_1),a(\mu_2)\}$ for arbitrary characters $\mu_1$, $\mu_2$, with equality guaranteed whenever $a(\mu_1)\neq a(\mu_2)$.
\end{proof}

\begin{example}
As a special case of Lemma \ref{lem_wal}, if a supercuspidal representation $\pi$ has the property that $a(\pi)$ is odd or $a(\pi)=2$,  then it is automatically minimal.
A well-known theorem of Atkin--Li \cite[Theorem 4.4]{atkin-li} gives a partial converse when $q=2$: a representation $\pi$ of $\GL_{2}(\Q_2)$ with $a(\omega_\pi)<a(\pi)/2$ is minimal if and only if $a(\pi)$ is odd or $a(\pi)=2$. \end{example}

\subsubsection{Newforms and the Whittaker model}

We shall work in the Whittaker model $\Wh(\pi,\psi)\isom\pi$ which carries the right-regular action of $G$ in the space of functions $W\colon G\rightarrow \C$ satisfying $W(z n(x)g)=\omega_{\pi}(z)\psi(x)W(g)$ for each $z\in Z$, $x\in F$, $g\in G$.

\begin{defn}\label{def:localwhit}

We call the unique $K_{1}(a(\pi))$-invariant vector $W_{\pi}\in\Wh(\pi,\psi)$ satisfying $W_{\pi}(1)=1$ the \textit{normalised Whittaker newform}.
\end{defn}

The following lemma records the value of $W_{\pi}$, for each $\pi$, on the toral elements $a(\varpi^{r})\in A$.

\begin{lemma}\label{lem_dough} Suppose $a(\pi)\geq 1$.

\begin{itemize}

\item If $\pi \simeq\chi_{1}\boxplus\chi_{2}$ with $a(\chi_{1})>0=a(\chi_{2})$, then
\begin{equation}\label{e:toral1}
W_{\pi}(a(\varpi^{r}))=\left\lbrace\begin{array}{cl}
(\chi_{1}(\varpi)q^{-1/2})^{r}& \text{if } r\geq 0\\
0& \text{if } r<0.
\end{array}\right.
\end{equation}

\item If $\pi \simeq\chi\St$ with $a(\chi)=0$, then
\begin{equation}\label{e:toral2}
W_{\pi}(a(\varpi^{r}))=\left\lbrace\begin{array}{cl}
(\chi(\varpi)q^{-1})^{r}& \text{if } r\geq 0\\
0& \text{if } r<0.
\end{array}\right.
\end{equation}

\item Else if $\pi$ satisfies $L(s,\pi)=1$, then
\begin{equation}\label{e:toral3}
W_{\pi}(a(\varpi^{r}))=\left\lbrace\begin{array}{cl}
1& \text{if } r=0\\
0& \text{if } r\neq 0.
\end{array}\right.
\end{equation}
\end{itemize}
\end{lemma}

\begin{proof}
The above formulae are well-known in the literature; for example, they appear verbatim in \cite[(121)]{transfer}. For completeness, we give a detailed proof here relying on the results in \cite{ralf-newforms}.

Let $W_{\tilde{\pi}}^*$ be the unique $K'_1(a(\pi))$-invariant vector in $\Wh(\tilde\pi,\psi)$ satisfying $W^*_{\tilde\pi}(1)=1$ (recall that $a(\pi) = a(\tilde \pi))$. As $\tilde\pi \cong \omega_\pi^{-1}\pi$, we may define a linear isomorphism  from $\Wh(\pi,\psi)$ to $\Wh(\tilde\pi,\psi)$  by $W \mapsto W'$ where $W'(g) = \omega_\pi^{-1}(\det (g)) W(g)$. The function $W_\pi'$ is  $K'_1(a(\pi))$-invariant, so by uniqueness, we have $W_{\tilde{\pi}}^*(g) = W_\pi'(g)= \omega_\pi^{-1}(\det (g)) W_\pi(g)$. This implies that
\begin{equation}\label{keyeqstar}W_\pi(a(\varpi^r)) = \omega_\pi(\varpi^r) W_{\tilde{\pi}}^*(a(\varpi^r)).\end{equation}

The values $W^*_{\tilde{\pi}}(a(y))$ in all cases were computed explicitly and written down in the table just before Section 3 of \cite{ralf-newforms}. We note here that the function $y \mapsto W^*_{\tilde{\pi}}(a(y))$ is the local newform in the Kirillov model of $\tilde{\pi}$ by the conventions of \cite{ralf-newforms}.

In particular, when $\pi \simeq\chi_{1}\boxplus\chi_{2}$ with $a(\chi_{1})>0=a(\chi_{2})$ then $\tilde \pi \simeq \chi_{1}^{-1}\boxplus\chi_{2}^{-1}$ and the table from \cite{ralf-newforms} gives us that $W_{\tilde{\pi}}^*(a(\varpi^r)) = \chi_2^{-1}(\varpi^r).$ Combined with \eqref{keyeqstar} this gives us \eqref{e:toral1}.
If $\pi \simeq\chi\St$ with $a(\chi)=0$, then $\tilde \pi \simeq \chi^{-1}\St$ and the table of \cite{ralf-newforms} gives us that $W_{\tilde{\pi}}^*(a(\varpi^r)) = q^{-r} \chi^{-1}(\varpi^r)$ if $r \ge 0$ and equal to 0 if $r<0$. Combined with \eqref{keyeqstar} this gives us \eqref{e:toral2}.
Finally if $L(s, \pi)=1$, then the table of \cite{ralf-newforms} gives us that $W_{\tilde{\pi}}^*(a(\varpi^r)) = 1$ if $r = 0$ and equal to 0 if $r\ne0$. Combined with \eqref{keyeqstar} this gives us \eqref{e:toral3}.
\end{proof}

\subsection{The vanishing index}\label{sec_evaluate_Whittaker}

Throughout the rest of \S\ref{s:local}, $\pi$ will be an irreducible, admissible, generic representation of $G$. We will denote $n=a(\pi)$ and $m=a(\omega_{\pi})$. The triple $(t,l, v)$ will always be so that $t\in \Z$, $0\le l \le n$, $v\in \OF^\times$. By Lemma 2.13 of \cite{saha-large-values} and by right-$K_{1}(n)$-invariance, the newform $W_{\pi}$ is completely determined by its values on the representatives $g_{t,l,v}$  of $G$; moreover, for each $t,l$, the map $v \mapsto W_\pi(g_{t,l,v})$ depends only on $v$ modulo $U_l$ and the map $v \mapsto |W_\pi(g_{t,l,v})|$ depends only on $v$ modulo $U_{\min(l, n-l)}$.

The following  proposition was proved in \cite[Prop. 2.10]{saha-hybrid}.

\begin{proposition}\label{dpiprop}  Define
$
d_{\pi}(l)=\max\{n, l+m, 2l\}
$. Suppose $W_\pi(g_{t,l,v}) \neq 0$. Then $t \ge - d_{\pi}(l).$
\end{proposition}
\begin{remark}The quantity $d_\pi(l) - l$ is invariant under the substitution $l \mapsto n-l$.
\end{remark}

\begin{defn}\label{def_local_ram_index}
For each $0\le l \le n$, the level $l$ \emph{vanishing index} $e_\pi(l)$ of $\pi$  is defined via
\begin{align*}
e_{\pi}(l)&=\min\left\lbrace\, r\geq 0\ :\ \exists v \in \OF^\times \text{ satisfying } W_{\pi}(g_{r-d_{\pi}(l),l,v})\neq 0\, \right\rbrace \\&=\min\left\lbrace\, r\geq 0\ :\ \int_{v\in\Oix} \abs{W_{\pi}(g_{r-d_{\pi}(l),l,v})}^{2}\,\dxv\neq 0\, \right\rbrace.
\end{align*}
\end{defn}
We now state our main local result.
\begin{theorem}\label{thm_andys_theorem}

Let $\pi$ be an irreducible, admissible, generic representation of $G$ with conductor $a(\pi)=n$. Let $l$ be an integer such that $0\leq l\leq n$. Then if $q>3$ we have $e_{\pi}(l)=0$.

If $q=3$, then $e_{\pi}(l)=0$, except in one case:

\begin{itemize}
\item $e_{\pi}(l)=1$ if

\begin{enumerate}

\item[(i)] $\pi \simeq\chi_{1}\boxplus\chi_{2}$ with   $a(\chi_{1})=a(\chi_{2})=l$, $n=2l\geq 4$, and $a(\chi_{1}\chi_{2}^{-1})=l$.

\end{enumerate}
\end{itemize}

If $q=2$, then $e_{\pi}(l)=0$, aside from the following exceptions:

\begin{itemize}

\item $e_{\pi}(l)=1$ when

\begin{enumerate}

\item[(ii)] $\pi \simeq\chi_{1}\boxplus\chi_{2}$ with $a(\chi_{1})$ and $a(\chi_{2})$ both at least 2, $a(\chi_{1}) \neq a(\chi_2)$, and $l=a(\chi_{1})$ or $l=a(\chi_{2})$;

\item[(iii)] $\pi \simeq\chi\pi_{0}$ and $n=2l\geq 4$ where $\pi_{0}$ is a supercuspidal representation with\footnote{As $n-1$ is odd, note that  $\pi_0$ is automatically minimal.} $a(\pi_{0})=n-1$ and $\chi$ is a  character of conductor $a(\chi)=n/2$.

\end{enumerate}

\item $e_{\pi}(l)=2$ when

\begin{itemize}

\item[(iv)] $\pi \simeq\chi\St$ with $a(\chi)\ge 2$ and $n=2l = 2 a(\chi)\geq 4$;

\item[(v)] $\pi \simeq\chi_{1}\boxplus\chi_{2}$ with $n=2l \ge 4$, $a(\chi_{1})=a(\chi_{2})=l$, $\chi_1 \chi_2^{-1} \notin \{| \cdot |, \ | \cdot |^{-1} \}$, and  $a(\chi_{1}\chi_{2}^{-1})<l-1$;

\item[(vi)] $\pi \simeq\chi\pi_{0}$ and $n=2l\geq 4$ where $\pi_{0}$ is a minimal supercuspidal representation with $a(\pi_{0})\leq n-2$ and $\chi$ is a character of conductor $a(\chi)=n/2$.
\end{itemize}

\item $e_{\pi}(l)=3$ when

\begin{itemize}

\item[(vii)] $\pi \simeq\chi_{1}\boxplus\chi_{2}$ with $n=2l\ge 6$, $a(\chi_{1})=a(\chi_{2})=l$ and $a(\chi_{1}\chi_{2}^{-1})=l-1$.

\end{itemize}

\end{itemize}
\end{theorem}

We give a proof of Theorem \ref{thm_andys_theorem} in \S \ref{sec_proof_of_thm}.

\subsubsection{Basic properties}

\begin{proposition}\label{basic1prop} The vanishing index $e_\pi(l)$ has the following properties:
\begin{enumerate}
\item For all unramified characters $\chi$ we have $e_\pi(l) = e_{\chi\pi}(l).$

\item We have $e_\pi(l) = e_{\tilde{\pi}}(n-l)$.
\end{enumerate}
\end{proposition}
\begin{proof}

The first assertion follows from the fact that $W_{\pi\chi} (g) = \chi(\det(g)) W_\pi(g)$ for all unramified characters $\chi$. For the second identity, we may first twist $\pi$ by an unramified character to ensure that $\omega_{\pi}(\varpi)=1$ (by part (1), this does not change the vanishing index). Moreover, $d_\pi(l) =  d_{\tilde \pi}(l)= d_{\tilde{\pi}}(n-l)+2l-n$. Now use the ``generalised Atkin-Lehner relation" of \cite[Prop. 2.28]{saha-large-values} that implies that $W_{{\pi}}(g_{t,l,v})$ is non-zero
if and only if $W_{\tilde\pi}(g_{t+2l-n,n-l,-v})
$ is non-zero. The result follows.
\end{proof}

\begin{proposition}\label{basic2prop}

Suppose that $n \le 1$. Then   $e_\pi(l) = 0$.

\end{proposition}

\begin{proof}

The case $n=0$ is trivial, since $W_\pi(1) \neq 0$. Now suppose $n=1$. Using part (2) of the previous Proposition, we may assume (by replacing $\pi$ by $\tilde{\pi}$ if necessary) that $l=1$. In this case, $d_{\pi}(1)=2$ and the matrix $g_{-2,1,1}$ lies in the double coset class $ZNK_1(1)$; this follows from the more general formula
\begin{equation}
a(y)=\begin{pmatrix}
y&\\
&1
\end{pmatrix}= z(-\varpi^{n})n(\varpi^{v(y)-n})g_{v(y)-2n,n,1} \arraycolsep=4pt\def\arraystretch{0.5}\begin{pmatrix}
1&\\
-\varpi^{n}&1
\end{pmatrix}a(\varpi^{-v(y)}y)
\end{equation}
for any $y\in\Fx$ and $n\geq 0$. It follows that $|W_{\pi}(g_{-2,1,1})|=1$ and therefore $e_\pi(1)=0$.
\end{proof}

\subsubsection{Uniform vanishing and rationality}\label{s:localrational}As it stands, the quantity $e_\pi(l)$ is characterized by the following properties:
\begin{enumerate}

\item $W_\pi(g_{t,l,v}) = 0$ for all $t < e_\pi(l) - d_\pi(l)$, and all $v \in \OF^\times$.
\item $W_\pi(g_{e_\pi(l) - d_\pi(l),l,v}) \neq 0$ for \emph{some} $v \in \OF^\times$.
\end{enumerate}

It would be nice if in the second condition above, we could replace ``some" by ``all". While this cannot be done in general, there are  indeed some  situations where this is possible. One such situation is when $n \le 1$, as then $e_\pi(l)=0$ and $|W_\pi(g_{- d_\pi(l),l,v})|$ (which depends only on $v$ modulo $U_{\min(l, n-l)}$) is non-zero for all $v$ since $0\le l \le 1$.  We now describe another such situation in the special case $F=\Q_p$.

\begin{proposition}\label{prop:localrational}

Suppose that $F = \Q_p$ and that there exists a complex number $s$ such that $\pi_s := \pi | \cdot |^s$ has the property that $\Q(\pi_s)$ is a number field. Suppose that for some $0 \le l \le n$, $\Q(\pi_s) \cap \Q(\mu_{p^{\min\{l, n-l\}}}) = \Q$. Then for any integer $t$, the following are equivalent:

\begin{enumerate}
\item $W_\pi(g_{t,l,v}) \neq 0$ for some $v \in \Zpx$.

 \item  $W_\pi(g_{t,l,v}) \neq 0$ for all $v \in \Zpx$.
\end{enumerate}
In particular, $W_\pi(g_{e_\pi(l) - d_\pi(l),l,v}) \neq 0$ for all $v \in \Zpx$.
\end{proposition}

\begin{proof}

Using the fact that $W_{\pi_s} (g) = |\det(g)|^s W_\pi(g)$, it follows that (by replacing $\pi$ by $\pi_s$) we may assume that $s=0$.

For each $\sigma \in \Aut(\C)$ let $t_\sigma \in \Zpx$ be the unique element such that $\sigma(\psi(x)) = \psi(t_\sigma x)$ for all $x \in \Q_{p}$. The map $\sigma \mapsto t_\sigma$ factors through $\Gal(\Q(\mu_{p^\infty}) /\Q)$ and, as is well-known, gives an isomorphism $\Gal(\Q(\mu_{p^\infty}) /\Q) \simeq \Z_p^\times$ (in fact, this is a special case of class field theory). The image of $\Gal(\Q(\mu_{p^\infty}) /\Q(\mu_{p^r}))$ under this isomorphism is precisely $U_{r}$. In particular, as $\sigma$ traverses  $\Gal(\Q(\pi)\Q(\mu_{p^{\min\{l, n-l\}}})/\Q(\pi)) \simeq \Gal(\Q(\mu_{p^{\min\{l, n-l\}}})/\Q)$, $t_\sigma$ traverses every coset in $\Z_p^\times/U_{\min\{l, n-l\}}$.

Next, let $m_\sigma = \mat{1}{}{}{t_\sigma}$. It is easy to check that the map $W \mapsto W'$ defined by $W'(g) := \sigma(W(m_\sigma g))$ is $\sigma$-linear and takes $\Wh(\pi, \psi)$ to $\Wh({}^\sigma\pi, \psi).$ Therefore, we get $W_{{}^\sigma\pi}(g) = \sigma(W_\pi(m_\sigma g)).$
In particular, for each $\sigma \in \Aut(\C/\Q(\pi))$, we have $W_{\pi}(g_{t,l,v}) \ne 0 \Leftrightarrow  W_\pi(g_{t,l,vt_\sigma})\neq0$. Letting $\sigma$ vary in $\Gal(\Q(\pi)\Q(\mu_{p^{\min\{l, n-l\}}})/\Q(\pi))$, we see that if $W_{\pi}(g_{t,l,v}) \ne 0$ for some $v$ in $\Z_p^\times/U_{\min\{l, n-l\}}$, then $W_{\pi}(g_{t,l,v}) \ne 0$ for \emph{all} $v$ in $\Z_p^\times/U_{\min\{l, n-l\}}$.
\end{proof}

\subsection{The proof of Theorem \ref{thm_andys_theorem}}\label{sec_proof_of_thm}

In this subsection, we prove Theorem \ref{thm_andys_theorem}. Thanks to Propositions \ref{basic1prop} and \ref{basic2prop}, we can and will make the following assumptions throughout this subsection:
 $n \ge 2$, $\omega_\pi \in \Xc$, $l \le \frac{n}2$.

\subsubsection{The basic identity}

We now review a powerful tool for computing the values $W_{\pi}(g_{t,l,v})$. For each $t,l$, the function on $v\in\Oix$ given by $v\mapsto W_{\pi}(g_{t,l,v})$ is well defined on the quotient $\Oix/U_l$. By Fourier inversion, for each $\mu\in\Xc_{l}$ there exists a Fourier coefficient $c_{t,l}(\mu)\in\C$ such that
\begin{equation}\label{eq_hazel}
W_\pi(g_{t,l,v})=\sum_{\mu\in\Xc_{l}}c_{t,l}(\mu)\,\mu(v).
\end{equation}
In \cite[\S 2]{saha-large-values} it was shown that one can mill down the zeta-integrals occurring in the local functional equation for $\GL_{2}$ to reveal a polynomial equation in  the $c_{t,l}(\mu)$: we call this \textit{the basic identity}.

\begin{prop}\label{prop_basic_identity} Assume that $\omega_{\pi}(\varpi)=1$. We have the following identity between
 polynomials in the variables $q^{s}$ and $q^{-s}$:

\begin{equation*}
\begin{array}{l}\vspace{0.1in}
\displaystyle\varepsilon(1/2,\mu\pi)\, \sum_{t\in\Z}\, q^{(t+a(\mu\pi))(1/2-s)}\,c_{t,l}(\mu)\,L(s,\mu\pi)^{-1}\hfill\\
\hspace{0.5in}\displaystyle=\omega_{\pi}(-1)\,\sum_{r\geq 0}\, \,\,q^{-r(1/2-s)}\,W_\pi(a(\varpi^r)) \, \Ga(\varpi^{r-l},\mu^{-1})\,L(1-s,\mu^{-1}\omega_{\pi}^{-1}\pi)^{-1}.
\end{array}
\end{equation*}
\end{prop}
\begin{proof}
This is proved explicitly in \cite[Prop.\ 2.23]{saha-large-values}. We briefly recall the proof. Let $W' = \pi(w.n(\varpi^{-k}))W_\pi$ and let the Jacquet--Langlands local zeta integrals $Z(W',s,\mu)$, $Z(w \cdot W',1-s,\mu^{-1}\omega_\pi^{-1})$  be defined as in \S2.5 of \cite{saha-large-values}. By the Jacquet-Langlands functional equation (see Theorem 2.21 of \cite{saha-large-values}) $\frac{Z(W',s,\mu)}{L(s,\mu\pi\subscriptp \twist )}\eps(s,\mu\pi\subscriptp \twist)$  is equal to $\frac{Z(w \cdot W',1-s,\mu^{-1}\omega_\pi^{-1})}{L(1-s, \pi\subscriptp \twist \mu^{-1}\omega_\pi^{-1})}$. On the other hand, using $\eps(s, \mu \pi) = \eps(1/2,\mu \pi)q^{a(\mu\pi)(\frac12-s)}$, a calculation (performed in detail in \cite[Sec. 2.6]{saha-large-values}) gives that the left side of the Proposition is equal to $\frac{Z(W',s,\mu)}{L(s,\mu \pi)}\eps(s,\mu \pi)$  and the right side is equal to $\frac{Z(w \cdot W',1-s,\mu^{-1}\omega_\pi^{-1})}{L(1-s, \pi\subscriptp \twist \mu^{-1}\omega_\pi^{-1})}$.
\end{proof}

\begin{rem}
The fact that each side the basic identity is an element of $\C[q^{s},q^{-s}]$ follows from the proof of Jacquet--Langlands' local functional equation \cite[Theorem\ 2.18]{jacquet-langlands}.
\end{rem}

\begin{rem}
Apart from the basic identity proved in \cite[\S2]{saha-large-values}, there exist other formulas for the Whittaker newvector
which are useful in various contexts. In particular, Templier in an
unpublished manuscript from 2011, obtained a formula that
expresses the Whittaker newvector in terms of a family of ${}_2F_1$
hypergeometric integrals, see also \cite[\S4]{templier-large}; a very similar formula was  obtained independently by Hu in 2016 (also unpublished but see \cite[Lemma 2.12]{hu-subconvex} for the case of principal series), and also independently by Assing \cite[\S3]{assing-local} in all cases. However, the basic identity is more suitable for the purpose of establishing the results in the present paper.
\end{rem}

Let us introduce the notation
\begin{equation}\label{eq_butter}
t_{\pi}(\mu,l)=\min\{ t\in\Z\,:\,c_{t,l}(\mu)\neq 0\}\in\Z\cup\{\infty\}
\end{equation}
where we say $t_{\pi}(\mu,l)=\infty$ if and only if $c_{t,l}(\mu)= 0$ for all $t\in\Z$. We have already noted that $c_{t,l}(\mu)=0$ for $t<-d_\pi(l)$. In terms of the expansion \eqref{eq_hazel}, for any $t\in\Z$ and $0\leq l \leq n$ we have
$\int_{v\in\Oix} \abs{W_{\pi}(g_{t,l,v})}^{2}\,\dxv=\sum_{\mu\in\Xc_{l}}
\,\abs{c_{t,l}(\mu)}^{2}$
by orthogonality of characters. So we have
\begin{equation} \label{keyeq2}
e_{\pi}(l)=\min\lbrace t_{\pi}(\mu,l)+d_{\pi}(l) \ : \ \mu\in\Xc_{l}\rbrace.
\end{equation}
\subsubsection{A case by case analysis}

We now compute $t_{\pi}(\mu,l)$ as $\pi$ varies over the types listed in \S \ref{sec_classification}, using the formula \eqref{keyeq2} to evaluate $e_\pi(l)$ in each instance. As $n \ge 2$, we only need to consider representations of types (2) and (3).

\subsubsection*{Type (2)}

Let $\pi \simeq\chi_{1}\boxplus\chi_{2}$ with $a(\chi_{1})>0$ and $a(\chi_{2})=0$. Since $m=n$ we have $d_{\pi}(l)=n+l$. It shall be sufficient to check the case of the trivial character $\mu=1$, which belongs to $\Xc_l$ for each $l$. We compare both sides of the basic identity given by Proposition \ref{prop_basic_identity}: on the left-hand side, the least non-zero exponent of $q^{-s}$ that appears is $t_{\pi}(1,l)+n$ while on the right-hand side we claim it is $-l$. Indeed, putting $X=q^{-s}$ and using the formula \eqref{e:toral1}, the right side of the basic identity is as follows:
$$\omega_{\pi}(-1)\,(1 - \chi_1(\varpi)q^{-1} X^{-1})\sum_{r\geq 0}\, \,\,q^{-r}X^{-r}\chi_1(\varpi)^r  \Ga(\varpi^{r-l},1).
$$
Using \eqref{eq1lem}, this simplifies to $$\omega_{\pi}(-1)\zeta(1)q^{-l}X^{-l}
\chi_1(\varpi^l)(1-\chi_1(\varpi^{-1})X),$$ which makes it clear that the least non-zero exponent of $X$ on the right side is $-l$.

We thus conclude $t_{\pi}(1,l) = -n-l = - d_{\pi}(l)$. This implies that
\begin{equation}
e_{\pi}(l)=\min\lbrace t_{\pi}(\mu,l)+d_{\pi}(l) \ : \ \mu\in\Xc_{l}\rbrace \le 0.
\end{equation}
Since $e_{\pi}(l)$ is a non-negative integer, it follows that $e_{\pi}(l)=0.$

\subsubsection*{Generalities on type (3)}

Let $\pi$ be any representation of type (3). Then $L(s, \pi)=1$ and $W_\pi(a(\varpi^r))$ is non-zero if and only if $r=0$. The basic identity reads
\begin{equation}\label{eq_roasted}
\omega_{\pi}(-1) \, \Ga(\varpi^{-l},\mu^{-1})\,L(1-s,\mu^{-1}\omega_{\pi}^{-1}\pi)^{-1}=\varepsilon(1/2,\mu\pi)\, \sum_{t\in\Z}\, q^{(t+a(\mu\pi))(1/2-s)}\,c_{t,l}(\mu)\,L(s,\mu\pi)^{-1}
\end{equation}
The least $t=t_{\pi}(\mu,l)$ for which such a $c_{t,l}(\mu)$ is non-zero depends on the support for the Gauss sum (Lemma \ref{lem_pine}) and the specific form of the (at most degree-two) $L$-factors, but we can make a few general remarks. In the sequel, put $X=q^{-s}$ for simplicity.

\begin{lemma}\label{lem_salted}

Let $\pi$ be of type (3).

\begin{enumerate}

\item If $l \le 1$ (which is always the case when $n\le 3$), then $t_{\pi}(1,l) = -n$. In particular, as $d_{\pi}(l) = n$, it follows that $e_{\pi}(l)=0$ whenever $l \le 1$.

\item If $l \ge 2$, then $t_{\pi}(\mu,l) = \infty$ unless $a(\mu)=l$, in which case $t_{\pi}(\mu,l) =  -\delta -a(\mu \pi)$, where $\delta\in\{0,1,2\}$ is the degree of the polynomial $L(s,\mu\pi)^{-1}$ in $q^{-s}$.

\end{enumerate}

\end{lemma}

\begin{proof}

In part (1) we need to consider $\mu=1$. Then $L(s,\mu\pi)=L(1-s,\mu^{-1}\omega_{\pi}^{-1}\pi)=1$ and combining \eqref{eq_roasted} with Lemma \ref{lem_pine} we have
\begin{equation*}
\omega_{\pi}(-1) \, (\zeta(1)q^{-1})^{\delta_{l,1}}=\varepsilon(1/2,\pi)\, \sum_{t}q^{(t+n)(1/2-s)}c_{t,l}(1)
\end{equation*}
implying $c_{-n,l}(\mu)$ is the only non-zero value of $c_{t,l}(\mu)$ in this case.

 For part (2), we begin by noting that $\Ga(\varpi^{-l},\mu^{-1})=0$ whenever $a(\mu)\neq l\geq 2$. So $t_{\pi}(\mu,l) = \infty$ unless $a(\mu)=l$, in which case the least exponent of $X$ occurring in the right side of \eqref{eq_roasted} is $t_{\pi}(\mu,l) + a(\mu \pi)$ and in the left side is $-\delta$. We deduce $t_{\pi}(\mu,l) =  -\delta -a(\mu \pi)$.
\end{proof}

Thus, after Lemma \ref{lem_salted}, we proceed by assuming that $l \ge 2$ (and hence $n \ge 4$) and furthermore that $a(\mu) = l$.

\subsubsection*{Type (3.a)}

Let $\pi \simeq\chi\St$ with $a(\chi)>0$. By twisting by an unramified character, we may assume that $\chi \in \Xc$. Here we note $m\leq n/2=a(\chi)$ so that $d_{\pi}(l)=n$. As explained above, we assume $l\geq 2$ and $a(\mu)=l$. First consider the case  $\mu \ne \chi^{-1}$; hence $L(s,\mu\pi)=1$. We get $t_{\pi}(\mu,l) = -a(\mu \pi)$ and hence
\begin{equation*}
 d_{\pi}(l)+t_{\pi}(\mu,l)=n-a(\mu\pi)=n-2a(\mu\chi).
\end{equation*}
So we need to look for those $\mu\in\Xc_{l}', \mu \neq \chi^{-1}$ that maximise $a(\mu\chi)$. Indeed if $l \neq n/2$, then $a(\mu\chi)=n/2$, implying $e_{\pi}(l)=0$.

So, we assume from now on that $l=n/2$, implying $a(\mu\chi) \le n/2$. By Lemma \ref{lem_monkey} we can always find $\mu\in\Xc_{l}'$ satisfying $a(\mu\chi)=l$ if and only if $q>2$; in this case again $e_{\pi}(l)=0$. If $q=2$, then by the same lemma we can choose $\mu\in\Xc_{l}'$ such that $a(\mu\chi)=l-1$ if and only if $l>2$, in which case $d_{\pi}(l)+t_{\pi}(\mu,l)=2$ and so $e_{\pi}(l)\leq 2$. If $q=2$ and $l=n/2=2$, then we have $a(\mu\chi)=0$ and hence  $d_{\pi}(l)+t_{\pi}(\mu,l) = n =4$.

Finally, we consider the case $\mu = \chi^{-1}$. Then $\delta=1$ and the basic identity gives us
\begin{equation*}
d_{\pi}(l)+t_{\pi}(\mu,l)=n-1-a(\mu\pi)=n-2
\end{equation*}
as $a(\mu\pi)=1$. This estimate falls short of $e_{\pi}(l)\leq 2$ whenever $n>4$ so we must in fact have $e_{\pi}(l)=2$ in those cases. If $n=2k=4$, then we have found $e_{\pi}(l)=2$.

So, to summarize, $e_\pi(l) =2$ if $l=n/2 \ge 2$ and equals 0 otherwise.

\subsubsection*{Type (3.b)} Let $\pi \simeq\chi_{1}\boxplus\chi_{2}$ with $a(\chi_{1})\geq a(\chi_{2})>0$. Here, $m=a(\chi_{1}\chi_{2})<n$ and $d_{\pi}(l)=\max\{n,m+l\}$. We shall need to divide our analysis into several cases.

Case 1: Suppose $a(\chi_{i})\neq l$ for both $i=1,2$. In this case $\delta=0$. Hence by Lemma \ref{lem_wal}, $t_{\pi}(\mu,l) = -a(\mu\pi) = -d_\pi(l)$ for all $\mu\in\Xc_{l}'$; consequently
$
e_{\pi}(l)=0.
$

Case 2: Suppose $a(\chi_{1}) > a(\chi_{2})=l.$ Then if $q>2$ there exists a $\mu\in\Xc_{l}'$ such that $a(\mu\chi_{2})=l$; in particular $a(\mu\chi_{i})>0$ so that $L(s,\mu\pi)=1$. Therefore we again get $e_{\pi}(l)=0$. If $q=2$ but $l>2$, then we can find a $\mu\in\Xc_{l}'$ such that $a(\mu\chi_{2})=l-1$ and this gives the maximum value for $a(\mu\pi)=l-1$; hence
$e_{\pi}(l)= d_{\pi}(l)-a(\chi_1)-l+1= 1.$
 If $q=l=2$, then we instead get $a(\mu\chi_{2})=0$; this implies that $\delta=1$  and
we once again get $e_\pi(l)=1$.

We have shown that if $a(\chi_{1}) > a(\chi_{2})=l$ then

\begin{equation*}
e_{\pi}(l)=\left\lbrace\begin{array}{ll}\vspace{0.07in}
0&\ifrm q>2\,\,{\rm or}\,\, l\leq1\\
1&\ifrm q=2 \,\,{\rm and}\,\, l\geq 2.
\end{array}\right.
\end{equation*}

Case 3: Suppose $a(\chi_{1})=a(\chi_{2})= l$ and either $q>3$ or $q=3$ and $a(\chi_{1}\chi_{2}^{-1})<l$. We have $d_{\pi}(l)=2l=n$. Lemma \ref{lem_monkeys}, part (1) implies that there exists a $\mu\in\Xc_{l}'$ such that $a(\mu\chi_{1})=a(\mu\chi_{2})=l$. We conclude  that $e_{\pi}(l)=0$.

Case 4: Suppose  $a(\chi_{1})=a(\chi_{2})= l$, $q=3$ and $a(\chi_{1}\chi_{2}^{-1})=l$. We have $d_{\pi}(l)=2l=n$.  Using Lemma \ref{lem_monkeys}, part (2), we see that the minimum value of $t_{\pi}(\mu,l)$ is $n-1$ and hence $e_{\pi}(l)=1$.

Case 5: Suppose  $a(\chi_{1})=a(\chi_{2})= l$, $q=2$ and $a(\chi_{1}\chi_{2}^{-1})<l-1$. In this case Lemma \ref{lem_monkeys}, part (3) tells us that $e_{\pi}(l)=2$ (note that if $l=2$ then $\delta =2$)

Case 6: Suppose  $a(\chi_{1})=a(\chi_{2})= l$, $q=2$ and $a(\chi_{1}\chi_{2}^{-1})=l-1$. In this case we have $l \ge 3$, and Lemma \ref{lem_monkeys}, part (4) tells us that $e_{\pi}(l)=3.$

\subsubsection*{Type (3.c)} Let $\pi$ be a supercuspidal representation of $G$. As before we can assume $a(\mu) = l\geq 2$. We have $m\leq n/2$ so $d_{\pi}(l)=n$.  For supercuspidal representations we always have $\delta=0$ so that
\begin{equation*}
e_{\pi}(l)=n + \min\{t_{\pi}(\mu,l):\mu\in\Xc_{l}\}= n - \max\{a(\mu\pi):\mu\in\Xc_{l}\}
\end{equation*}
 By Lemma \ref{lem_wal}, if $\pi$ is minimal or $n\neq 2l$, then $a(\mu\pi)=n$; hence $e_{\pi}(l)=0$ in these cases. Otherwise suppose $n=2l$ is even and that $\pi_{0}$ is the minimal supercuspidal representation such that $\pi\cong\chi\pi_{0}$ with $\chi\in\Xc$. We must have $n>a(\pi_{0})$, by assumption on $\pi$, and $a(\chi)=l=n/2$, by the minimality of $\pi_{0}$. Moreover, by Lemma \ref{lem_wal} we have
$a(\mu\pi)=\max\{a(\pi_{0}),2a(\mu\chi)\}.$
If $q>2$, then there exists a $\mu\in\Xc_{l}'$ such that $a(\mu\chi)=l$. Therefore $e_{\pi}(l)=0$ if $q>2$. If $q=2$ but $l>2$, then there exist a $\mu\in\Xc_{l}'$ such that $a(\mu\chi)=l-1$ implying
\begin{equation*}
 e_{\pi}(l)=  2 - \max\{a(\pi_0) - n+2, 0\}.
\end{equation*}
Finally, if $n=2l=4$, then $a(\mu\chi)=0$. Thus
$e_{\pi}(2)= 4-a(\pi_{0})$.
\section{Global results}
In this section, we work over the field $\Q$ for simplicity; the modifications required for a  number field are straightforward. Also, while we stick to holomorphic newforms, one could easily write down corresponding results for Maass newforms.

\subsection{Adelisation of modular forms}

Let $\A$ denote the ring of adeles over $\Q$, let $S_{\f}$ denote the set of rational primes, that is the finite places of $\Q$, and let $\infty$ denote the real place. We put $\hat \Z = \prod_{p\in S_{\f}}\Z_p$. For any place $v$ of $\Q$, and any $g \in \GL_2(\Q)$, let $\iota_{v}(g)$ be the element of $\GL_2(\A)$ whose $v$th place equals $g$ and all other places equal 1; thus $\iota_{v}$ is given by the embedding $\GL_{2}(\Q)\emb \GL_{2}(\Q_v) \emb \GL_{2}(\A)$.  We let $\iota_{\f}=\otimes_{p<\infty}\iota_{p}$, so that $\iota_{\f}(g)$ is the element of $\GL_2(\A)$ which equals $g$ at all finite places and equals the identity at the infinite place. More generally, for any $g \in \GL_2(\A)$, we use $\iota_{\f}(g)$ to denote the element of $\GL_2(\A)$ whose $p$'th component equals $g_p$ if  $p \in S_{\f}$ and whose infinite component equals the identity.

 Let $\psi \colon \Q \backslash \A \rightarrow \C^\times$ be the additive character defined by $\psi=\prod_{v}\psi_{v}$ where $\psi_{\infty}(x) = e(x)$ if $x \in \R$ and $\psi_{p}(x) = 1$ for $x \in \Z_p$.  Let $\pi \simeq \otimes_v \pi_v$ be an irreducible, unitary, cuspidal automorphic representation of $\GL_{2}(\A)$ with central character $\omega_\pi = \prod_v \omega_{\pi_v}$, which we assume to be trivial on $\R_{>0}$. We can (and shall) realise $\pi$ as a subspace of the space of square-integrable, cuspidal automorphic forms on $\GL_{2}(\A)$. Let $\chi$ be the Dirichlet character associated to $\omega_\pi$. For each $p \in S_{\f}$, let $n_p = a(\pi_p)$, $m_p =a(\omega_{\pi_p})$; we put $N = \prod_p p^{n_p}$, $M= \prod_p p^{m_p}$. Let $k \ge 2$ be an integer, and assume that $\pi_\infty$ is the holomorphic discrete series representation  of lowest weight $k$. In other words, $\pi_\infty$ is the unique irreducible subrepresentation of $| \cdot |^{\frac{k-1}2} \sgn^k \boxplus |\cdot |^{\frac{1-k}2}$; note that $\omega_{\pi_\infty} = \sgn^k.$

Let $K_1(N) = \prod_{p \in S_{\f}} K_{1,p}(n_p) = \prod_{p \nmid N} \GL_{2}(\Z_p) \prod_{p\mid N}  K_{1,p}(n_p)$ be a standard congruence subgroup of $\GL_{2}(\hat{\Z}) = \prod_{p \in S_{\f}}\GL_{2}(\Z_p)$. Above, $K_{1,p}(n_p)$ is the group defined in \eqref{e:defk1}; throughout this section, we will use subscripts to denote previously defined local objects. We have the diagonal realisation
\begin{equation}
\Gamma_1(N)=K_1(N)\GL_{2}(\R)^+ \cap \GL_{2}(\Q).
\end{equation}
 Let $K_\infty = \SO_2(\R)$, which is a  maximal compact subgroup of $\GL_{2}(\R)^+$. We say that a non-zero automorphic form $\phi \in \pi$ is an adelic newform if $\phi$
is $K_1(N)$-invariant and satisfies
\begin{equation}\label{weighteq}\phi\left(g \mat{\cos(\theta)}{\sin(\theta)}{-\sin(\theta)}{\cos(\theta)} \right) = e^{ik \theta}\phi(g)
\end{equation}
for all $g \in \GL_{2}(\A)$. It is well-known that an adelic newform $\phi$ exists and is unique up to multiples, and
corresponds to a factorizable vector $\phi= \otimes_v \phi_v$.

If $\phi$ is an adelic newform, then the function $f$ on $\H$ defined by
\begin{equation}\label{e:deadeli}
f(gi) = \det(g)^{-k/2}j(g, i)^k\phi(g)
\end{equation}
for each $g \in \GL_2(\R)^+$ is a classical newform (see \cite{Li79}) of weight $k$, level $N$, and character $\chi$. The map \eqref{e:deadeli} from adelic to classical newforms is a bijection. Indeed, given a classical newform $f$ of  weight $k$, level $N$, and character $\chi$, one has the procedure of adelisation (see \S3.1 of \cite{sahapet}) that produces an automorphic form $\phi_f$ on $\GL_2(\A)$ that is $K_1(N)$-invariant and such that \eqref{weighteq} and \eqref{e:deadeli} hold (with $\phi$ replaced by $\phi_f$). This automorphic form $\phi_f$ is the adelic newform inside an irreducible cuspidal  automorphic representation $\pi$ which has the properties described earlier.

Finally, for any elements $x$, $y$ lying in some ring, we define the matrices $n(x) = \mat{1}{x}{}{1}$, $a(y) = \mat{y}{}{}{1}$. For $z = x+iy \in \H$, put $g_z = n(x)a(y) = \mat{y}{x}{0}{1} \in \GL_{2}(\R)^{+}$. So $g_{z}i=z\in\Hb$ and we have for all $\sigma \in \SL_2(\Q)$ the equalities
\begin{equation}\label{eq:f-to-g}
(f\vert_{k}\sigma^{-1})(z)=(f\vert_{k}\sigma^{-1} g_{z})(i)= y^{-k/2}\phi_f(\iota_{\infty}(\sigma^{-1})g_{z})=
y^{-k/2}\phi(g_{z}\iota_{\f}(\sigma)).
\end{equation}

\subsection{Whittaker and Fourier expansions}
\subsubsection{Adelic Whittaker expansion}
We call a function $W\colon \GL_{2}(\A)\rightarrow\C$ a $\psi$-Whittaker function if $W$ satisfies $W(n(x)g)=\psi(x)W(g)$ for each $x\in\A$ and $g\in \GL_{2}(\A)$. By the existence and uniqueness of Whittaker models for $\GL_{2}$, there exists a (unique) subspace of such functions which, under the right-regular action of $\GL_{2}(\A)$, is isomorphic to $\pi$; this subspace is called the Whittaker model and is denoted $\Wh(\pi,\psi)$. This $\GL_{2}(\A)$-isomorphism may be explicated as the map
\begin{equation*}
\phi(g)\longmapsto W_{\phi}(g)=\int_{\Q\bs\A}\phi (n(x)g) \,\overline{\psi(x)}\,dx
\end{equation*}
where we take the invariant probability measure $dx$ on $\A$.  By Fourier inversion, we can derive a Fourier expansion for $\phi$ in terms of $W_{\phi}$,
\begin{equation}\label{eq:whittaker-expn}
\phi(g)=\sum_{\xi\in\Qx}W_{\phi}(a(\xi)g).
\end{equation}
\subsubsection{Classical Fourier expansion}
We will now explicate the relation between the adelic Whittaker expansion and the classical Fourier expansion. We begin with the following lemma.

\begin{lemma}\label{lemma:fe}Let $\phi$ be an  automorphic form in the space of $\pi$ that satisfies \eqref{weighteq} and let $\delta \in \Z$ be such that $\phi$ is right invariant by $\iota_{\f}(n(\delta u))$ for all $u \in \hat \Z$.    Let $h$ be the holomorphic function on $\H$ defined by the equation $$h(z) = j(g_z, i)^k \phi(g_z).$$ Then $h$ has a Fourier expansion given by $$h(z) = \sum_{n>0}a_h(n)e^{\frac{2 \pi i nz}{\delta}}.$$     Moreover, for all $\xi \in \Q^\times$, we have
\begin{equation*}
W_{\phi}(a(\xi) g_{z})=\begin{cases}y^{k/2}\, a_{h}(n)\,e(nz/\delta)&\text{if } \xi=n/\delta\text{ for some }n\in \Z,\\
0& \text{otherwise.} \end{cases}
\end{equation*}
\end{lemma}
\begin{proof}
This proposition is given in \cite[Lemma 3.6]{gelbart} in a special case. Our statement is more general, thus we give a detailed proof. We have $h(z+ \delta) = y^{-k/2} \phi(g_z \iota_{\f}(n(-\delta))) = h(z)$ and as $h$ is a holomorphic cusp form  (of weight $k$ with respect to some principal congruence subgroup), it follows that $h$ has a Fourier expansion of the type specified. Next note that
\begin{equation*}
\begin{array}{rcl}\vspace{0.1in}
W_{\phi}(a(\xi) g_{z})&=&\displaystyle\int_{\Q\bs\A}\phi(n(x')a(\xi)g_{z}
)\psi(-x')\,dx'\\\vspace{0.1in}
&=&\displaystyle\int_{\Q\bs\A}\phi(n(x'+x)a(y))\psi(-\xi x')\,dx'\\\vspace{0.1in}
&=&e(\xi x) I_{\phi}(y,\xi),
\end{array}
\end{equation*}
where
\begin{equation*}
I_{\phi}(y,\xi)=\int_{\Q\bs\A}\phi(n(x')a(y))\psi(-\xi x')\,dx'.
\end{equation*}
  Choose the following fundamental domain for the (compact) quotient $\Q\bs\A$ using strong approximation:
$\Q\bs\A=[0,\delta)\times \prod_{p<\infty}\delta \Zp.$

This gives us
$$
I_{\phi}(y,\xi)=\int_{0}^{\delta}y^{k/2}h(x_{\infty}+iy)e\left(-\xi x_{\infty}\right)\,dx_{\infty}\left(\frac1{\delta}
\prod_{p<\infty}\int_{\Z_p}\psi_p (-\xi \delta x_p) dx_{p}\right).$$ If $\delta \xi$ is not an integer, then let $p$ be any prime dividing its denominator. We have $\int_{\Z_p}\psi_p (-\xi \delta x_p) dx_{p}$ $= 0$ and so $I_{\phi}(y,\xi)=0.$ On the other hand, if $\xi = n\delta$, then $\prod_{p<\infty}\int_{\Z_p}\psi_p (-\xi \delta x_p) dx_{p} = 1$ and we have
\begin{equation*}
\begin{array}{rcl}\vspace{0.1in}
I_{\phi}(y,n/\delta)&=&\displaystyle
\frac{1}{\delta}\int_{0}^{\delta}y^{k/2}h(x_{\infty}+iy)e\left(-\xi x_{\infty}\right)\,dx_{\infty}\\
&=&y^{k/2}\,a_h(n)\,e\left(niy/\delta\right).
\end{array}
\end{equation*}
where we have used the Fourier expansion of $h$. Therefore in this case we have $$W_{\phi}(a(n/\delta) g_{z})= e(nx/\delta)I_\phi(y, n/\delta ) = y^{k/2}\, a_{h}(n)\,e(nz/\delta)$$ as required.
\end{proof}
Henceforth, let $\phi$ be an adelic newform in the space of $\pi$ and $f$ the corresponding classical newform. We have the usual Fourier expansion for $f$ at infinity given by
\begin{equation}\label{fourier4}
f(z) = \sum_{n>0}a_f(n)e^{2 \pi i nz},
\end{equation}
and henceforth we normalise $f$ and $\phi$ so that $a_f(1) = 1$. More generally, for any cusp $\a = \frac{a}{L}$ with $(a,N)=1$, we put $\delta(\a) = \frac{\lcm[L^2, N, LM]}{L^2}$, and as explained in the introduction, we have the Fourier expansion of $f$ at $\a$:
\begin{equation}\label{fourier3}
(f|_k\sigma^{-1})(z) = \sum_{n>0}a_f(n;\a)e^{\frac{2 \pi i nz}{\delta(\a)}}.
\end{equation}
Above, $\sigma$ is any matrix in $\SL_2(\Z)$ such that $\sigma(a/L) = \infty$. In particular, we may choose $\sigma = \mat{a}{b}{L}{d}^{-1}$ where $b,d$ are any integers such that $ad-bL=1$. Note that  $a_f(n) = a_f(n,  \frac1N)$, which follows by taking $\sigma=\mat{1}{}{-N}{1} \in \Gamma_1(N)$ (note also that $\delta(\infty)= \delta(1/N)=1$). When $\delta(\a)>1$, the Fourier coefficient $a_f(n , \a)$ depends not just on $\a$ but also (weakly) on the choice of $\sigma$; precisely, for two choices $\sigma$, $\sigma'$ both taking $\a$ to $\infty$, the corresponding Fourier coefficients $a_f(n , \a)$ and $a'_f(n , \a)$ are related via $a_f(n , \a)= e^{\frac{2 \pi i nt}{\delta(\a)}}a_f'(n , \a)$ where $t$ is some integer depending on $\sigma' \sigma^{-1}$.  However, the absolute value $|a_f(n, \a) |$ is independent of the choice of $\sigma$.

 From \eqref{eq:f-to-g}, we see that the modular form $f\vert_{k}\sigma^{-1}$ corresponds to the automorphic form $\phi'=\pi(\iota_{\f}(\sigma))\phi$. Let $x$ be a finite adele such that $ x/ \delta(\a) \in \hat \Z$. Using the explicit formula \begin{equation}\label{eq:formulatransform}\mat{a}{b}{L}{d} \mat{1}{x}{}{1}\mat{a}{b}{L}{d}^{-1} =  \mat{1-xaL}{a^2x}{-L^2x}{1+xaL},\end{equation} we see that our adelic newform $\phi$ is right invariant by $\sigma^{-1} n(x)\sigma$ and consequently the automorphic form $\phi'$ is right invariant by $ n(x)$.

\begin{prop}\label{prop:whittaker-coeffs}
With notation as above, let $\xi\in\Qx$. Then
\begin{equation*}
W_{\phi}(a(\xi) g_{z}\iota_{\f}(\sigma))=\left\lbrace\begin{array}{cl}\vspace{0.05in}
y^{k/2}\, a_{f}(n;\a)\,e\left(nz/\delta(\a)\right)&\text{if } \xi=n/\delta(\a)\text{ for some }n\in \Z,\\
0& \text{otherwise}
\end{array}\right.
\end{equation*}
\end{prop}
\begin{proof}
This follows from applying Lemma \ref{lemma:fe} on the automorphic form $\phi' = \pi(\iota_{\f}(\sigma))\phi$, and using the fact that $\phi'$ is right invariant by $\iota_{\f}(n(\delta(\a) u))$ for all $u \in \hat \Z$.
\end{proof}

\subsubsection{A product formula for classical Fourier coefficients}\label{sec:product-formula-coefficients}
We can use Proposition \ref{prop:whittaker-coeffs} to pin-down the coefficients $a_{f}(n;\a)$ precisely in terms of the factorisation of $W_{\phi}$ into local Whittaker functionals. Indeed, by the uniqueness of local and global Whittaker models, we have for all $g \in G(\A)$, \begin{equation}\label{eq:W-factorisation}
W_{\phi}(g)=W_\infty(g_\infty)\prod_{p<\infty} W_{\pi_{p}}(g_p),
\end{equation} where, for each $p\in S_{\f}$, the function $W_{\pi_p}$ on $\GL_2(\Q_p)$ is defined as in Definition \ref{def:localwhit} while $W_{\infty}$ corresponds to $\phi_\infty$ in the Whittaker model for $\pi_\infty$, and is normalised so that \eqref{eq:W-factorisation} holds. To explicate the function $W_\infty$, we observe using Proposition \ref{prop:whittaker-coeffs} that for any $y>0$,
\begin{equation*}
W_{\infty}(a(y))=W_{\phi}(a(y))=y^{k/2}e^{-2 \pi y}.
\end{equation*}

For each prime $p|N$, and each  $0 \le l_p \le n_p$, let the integer $d_{\pi_p}(l_p)$ be defined as in Proposition \ref{dpiprop}, and let the integer $e_{\pi_p}(l_p)$ be as defined in \eqref{def_local_ram_index} (and written down explicitly in Theorem \ref{thm_andys_theorem}). The next Proposition rewrites the Fourier coefficients $a_f(n;\a)$ in terms of the local Whittaker newforms $W_{\pi_p}$.
\begin{prop}\label{propformulafourier}
Let  $\a=\sigma^{-1}\infty$ be a cusp of $\Gamma_0(N)\bs \H$ with $\sigma^{-1}=\left(\begin{smallmatrix}
a&b\\L&d
\end{smallmatrix}\right)\in\SL_{2}(\Z)$, $L|N$, $(a,N)=1$, so that $\a$ corresponds to the cusp $a/L$.  Let $r$ be a positive integer and write $r=r_{0}\prod_{p\mid N}p^{r_{p}}$ with $(r_{0},N)=1$, and write $L=\prod_{p\mid N}p^{l_{p}}$.  For each $p|N$ define
\begin{equation*}
u_{p}=-a\times\dfrac{p^{r_{p}}}{r}\times\dfrac{[L,M,N/L]}{p^{d_{\pi_{p}}(l_{p})-l_{p}}}\in\Zpx.
\end{equation*}
Then we have
\begin{equation}\label{e:formfour}
\dfrac{a_{f}(r;\a)}{r^{k/2}}=\dfrac{a_{f}(r_{0})}{r_{0}^{k/2}}\,e\left(\dfrac{rd}{\delta(\a)L}\right)\dfrac{\prod_{p\mid N}W_{\pi_p}(g_{r_{p}-d_{\pi_{p}}(l_{p}),l_{p},u_{p}})}{\delta(\a)^{k/2}}
\end{equation}
where $W_{\pi_p}$ are the local Whittaker newforms associated to $f$ normalised so that $W_{\pi_p}(1)=1$.
\end{prop}
\begin{proof}
Applying \eqref{eq:W-factorisation} to Proposition \ref{prop:whittaker-coeffs} (taking $z=i$) we determine
\begin{equation*}\begin{array}{rcl}\vspace{0.1in}
a_{f}(r;\a)
&=&\displaystyle e\left(\dfrac{ri}{\delta(\a)}\right)\, W_{\infty}(a(r/\delta(\a)))\prod_{p<\infty}W_{\pi_p}(a(r/\delta(\a))\sigma)\\
&=&\displaystyle a_{f}(1)\,\bigg(\dfrac{r}{\delta(\a)}\bigg)^{k/2}\prod_{p\mid r_{0}}W_{\pi_p}(a(r_{0}))\ \prod_{p\mid N}W_{\pi_p}(a(r/\delta(\a))\sigma).
\end{array}
\end{equation*}
Considering the $r_{0}$-th Fourier coefficient in the expansion at $\a=\infty$, (so that $\delta(\infty)=1$, and each $W_{\pi_p}$ is right invariant by $\iota_p(\sigma)$); this formula simplifies to
\begin{equation*}
a_{f}(r_{0})=a_{f}(1)r_{0}^{k/2}\prod_{p\mid r_{0}}W_{\pi_p}(a(r_{0}))
\end{equation*}
since $W_{\pi_p}(a(r_{0}))=1$ for $p\nmid r_{0}$. Comparing the formulas for $a_{f}(r;\a)$ and $a_{f}(r_{0})$, we see
\begin{equation}\label{e:initial}
a_{f}(r;\a)=\dfrac{a_{f}(r_{0})}{r_{0}^{k/2}}\,\left(\dfrac{r}
{\delta(\a)}\right)^{k/2}\,\prod_{p\mid N}W_{\pi_p}(a(r/\delta(\a))\sigma).
\end{equation}

Let us decompose this expression on the right. First consider the Bruhat decomposition of $\sigma$:
\begin{equation*}
\sigma=\begin{pmatrix}
d&-b\\-L&a
\end{pmatrix}=z(L)n(-d/L)a(1/L^{2})wn(-a/L).
\end{equation*}
This gives us, for each $p|N$,
\begin{equation}\label{eq:threefourths}\begin{split}
W_{\pi_p}\left(a(r/\delta(\a))\sigma\right)
&=\omega_{\pi_p}(L)\,\psi_{p}\left(\dfrac{-rd}
{\delta(\a)L}\right)\,W_{\pi_p}
\left(a\left(\dfrac{r}{L^{2}\delta(\a)}\right)w\,n\left(\dfrac{-a}{L}\right)\right)
\\&=\omega_{\pi_p}(L)\,\psi_{p}\left(\dfrac{-rd}{\delta(\a)L}\right)W_{\pi_p}\left(a(p^{r_{p}-d_{\pi_{p}}(l_{p})})\,w\,
\mat{1}{-a/L}{}{\frac{rp^{d_{\pi_{p}}(l_{p})}}{L^2 \delta(\a)p^{r_p}}} \right) \\&= \omega_{\pi_p}(L)\,\psi_{p}\left(\dfrac{-rd}{\delta(\a)L}\right) W_{\pi_p}\left(a(p^{r_{p}-d_{\pi_{p}}(l_{p})})\,w\,n
\left(u_{p} p^{-l_{p}}\right)\right).\end{split}
\end{equation}
where we have used the right-$K_{1,p}(N)$-transformation property of $W_{\pi_p}$. Substituting \eqref{eq:threefourths} into \eqref{e:initial}, we obtain
\begin{equation}
\dfrac{a_{f}(r;\a)}{r^{k/2}}=\left(\prod_{p\mid N}\psi_{p}\left(\dfrac{-rd}
{\delta(\a)L}\right) \omega_{\pi_p}(L) \right) \dfrac{a_{f}(r_{0})}{r_{0}^{k/2}}\,\dfrac{\prod_{p\mid N}W_{\pi_p}(g_{r_{p}-d_{\pi_{p}}(l_{p}),l_{p},u_{p}})}{\delta(\a)^{k/2}}.
\end{equation}
Our desired result \eqref{e:formfour} now follows from  the equalities $$\prod_{p\mid N}\psi_{p}(-rd/(\delta(\a)L))= \prod_{p<\infty} \psi_{p}(-rd/(\delta(\a)L))= e(rd/\delta(\a)L)$$ and $$\prod_{p\mid N}\omega_{\pi_p}(L)= \prod_{p<\infty}\omega_{\pi_p}(L)= 1.$$
\end{proof}
\subsubsection{Global to local reduction of the vanishing index}\label{sec:global-to-local-index}
\begin{corollary}\label{keycor}
Let $L=\prod_p p^{l_p}$, $r= r_0\prod_{p|N}p^{r_p}$ be positive integers with $L|N$, $(r_0, N)=1$.

\begin{enumerate}
\item If $r_p < e_{\pi_p}(l_p)$ for some prime $p|N$, then $a_f(r, \a) = 0$ for all cusps $\a = \frac{a}{L}$, $(a,N)=1$.
\item Let $r = \prod_{p|N} p^{e_{\pi_p}(l_p)}$. Then there exists some cusp $\a = \frac{a}{L}$, $(a,N)=1$, such that $a_f(r, \a) \neq  0$.

\end{enumerate}

\end{corollary}

\begin{proof}
Both parts follow immediately from Proposition \ref{propformulafourier}. Indeed, if $r_p < e_{\pi_p}(l_p)$ for some prime $p|N$, then $W_{\pi_p}(g_{r_p - d_{\pi_p}(l_p), l_p, v}) = 0$ for all $v \in \Z_p^\times$ and so \eqref{e:formfour} implies that $a_f(r;\a) = 0$. On the other hand, if $r = \prod_{p|N} p^{e_{\pi_p}(l_p)}$, then by the definition of $e_{\pi_p}(l_p)$, there exists for each $p|N$ some $v_p \in \Z_p^\times$ such that $W_{\pi_p}(g_{r_p - d_{\pi_p}(l_p), l_p, v_p}) \neq 0$. By the Chinese remainder theorem, we can now choose some $a$ coprime to $N$ satisfying $-a \times \frac{p^{r_p+l_p}}{Lr} \times \frac{[L^2, N, LM]}{p^{d_{\pi_p}(l_p)}} \equiv  v_p  \pmod{N}$ for all $p|N$. Therefore, by \eqref{e:formfour}, we have $a_f(r, \a) \neq  0$ (recall that $a_f(1) \neq 0$).
\end{proof}

Define the quantities
$$e_f(\a) :=  \min\{n>0: a_f(n;\a) \neq 0 \}, \quad e_f(L) = \min_{a \in (\Z/N\Z)^\times} e_f(a/L).$$
We can now prove our main result (stated as Theorem \ref{t:main2} in the introduction).

\begin{theorem}
For any integer $L= \prod_p p^{l_p}$ dividing $N$, we have $$e_f(L) = \prod_p p^{e_{\pi_p}(l_p)}.$$
\end{theorem}

\begin{proof}
The proof follows immediately from Corollary \ref{keycor}.
\end{proof}

\subsection{Proofs of Proposition \ref{prop:rational} and Theorem \ref{t:main}}\label{s:globalrational}

We begin by proving (a slightly stronger version of) Proposition \ref{prop:rational}. Let $\pi' = \pi$ if $k$ is even and $\pi' = \pi |\cdot |^{1/2}$ if $k$ is odd (here, $| \cdot |$ denotes the adelic norm, which is just the product of all the local norms). Then, as noted in \cite{raghutanabe}, $\pi'$ is a regular algebraic cuspidal automorphic representation (note however, that $\pi'$ is no longer unitary if $k$ is odd). In particular, if we define $\pi'_p = \pi_p$ if $k$ is even and $\pi'_p = \pi_p | \cdot |^{1/2}$ if $k$ is odd, so that $\pi' = \otimes_v \pi'_v$, then the compositum $\Q(\pi')$ of all the fields $\Q(\pi'_p)$ with $p \in S_{\f}$ is a number field (see \cite{raghutanabe}). In fact $\Q(\pi')=\Q(f)$ where $\Q(f)$ is the number field generated by all the Fourier coefficients $a_f(n)$ (see part (5) of Theorem 1.4 of \cite{raghutanabe}).

Let us define $K_N$ to be the compositum of all the fields $\Q(\pi'_p)$ over the primes $p$ such that $p^2|N$. Clearly $K_N$ is a subfield of $\Q(f)$.

\begin{proposition}\label{prop:globalrational}
Suppose for some divisor $L=\prod_p p^{l_p}$ of $N$ we have $K_N \cap \Q(e^{\frac{2 \pi i}{(L, N/L)}}) = \Q$. Then $e_f(\a) = e_f(L)$.
\end{proposition}

\begin{proof}
Using Proposition \ref{propformulafourier}, it suffices to show that $W_{\pi_p}(g_{e_{\pi_p}(l_p) - d_{\pi_p}(l_p), l_p, v}) \neq 0$ for all $v \in \Z_p^\times$. If $p^2 \nmid N$, this follows from Proposition \ref{basic2prop}. If $p^2|N$, this follows from Proposition \ref{prop:localrational}.
\end{proof}

Note that the above Proposition implies Proposition \ref{prop:rational} as $K_N\subseteq \Q(f)$.
Next, let $f$, $\pi$ be such that $f$ is the newform associated to an elliptic curve $E$ over $\Q$. In particular, $f$ has trivial character ($M=1$), $k=2$, and the Fourier coefficients $a_f(n)$ are all rational numbers.

\begin{lemma}
In the present case, the $e_p$ of Theorem \ref{t:main} is equal to the $e_{\pi_p}(l_p)$ of Theorem \ref{t:main2}.
\end{lemma}

\begin{proof}
Since the $a_f(n)$ are all rational numbers, the representation $\pi$ satisfies the hypothesis of Proposition \ref{prop:globalrational}. We consequently have $e_f(\a)=e_f(L)$ for any cusp $\a$ of denominator $L$, in particular for the $f$ attached to $E$.
\end{proof}

We now give the proof of how Theorem \ref{t:main2} implies Theorem \ref{t:main}. The fact that $f$ now has trivial character implies that if $\pi_p$ is a principal series representation, then it must be of the form $\chi \boxplus \chi^{-1}$. Furthermore it is known (see, e.g., \cite[Theorem 10.4]{silverman-advanced}) that the exponents $n_p = v_p(N)$ have the following bounds: $n_p \le 2$ if $p \ge 5$, $n_3 \le 5$, and $n_2 \le 8$. Now, one is left to rattle through the short (finite) list of possible entrants into Theorem \ref{t:main} (that is, which representations occur that satisfy the aforementioned bounds), with index $e_{p}$ given in Theorem \ref{t:main2}. Indeed, let $p=3$. Then, by Theorem \ref{t:main2}, and the bound $n_3 \le 5$, we see that $e_3 = 0$ unless $n_3 = 4$, $l_3=2$ and $\pi_3$ is a principal series representation, which is exactly Case (i) of Theorem \ref{t:main}. Next, let $p=2$. Case (ii) of Theorem \ref{t:main2} cannot occur in our present situation because $f$ has trivial character. Case (iii) of Theorem \ref{t:main2} exactly corresponds to Case (iii) of Theorem \ref{t:main}. Case (vii) of Theorem \ref{t:main2} corresponds to Case (iv) of Theorem \ref{t:main} (observe that there is no character $\chi$ of $\Q_2^\times$ satisfying $a(\chi) = 3$, $a(\chi^2)=2$, so we cannot have $n_3=6$). If we are not in any of the Cases (ii--vii) of Theorem \ref{t:main2}, then  we must have either $n_2 \le 2$ or $n_2 \ne 2 v_2(L)$, which corresponds to Case (ii) of Theorem \ref{t:main}. Finally, Cases (iv--vi) of Theorem \ref{t:main2} correspond to Case (v) of Theorem \ref{t:main}. This completes the proof that Theorem \ref{t:main2} implies Theorem \ref{t:main}.

\subsection*{Acknowledgements}

A.S. thanks Fran\c{c}ois Brunault, Paul Nelson, and Ameya Pitale for many interesting and useful discussions in July 2012 on the topic of this paper. The authors thank Kevin Buzzard, Kestutis Cesnavicius and Ralf Schmidt for feedback, as well as the anonymous referee for many useful suggestions which have improved the paper.

\bibliographystyle{amsalpha}	
\bibliography{refs-vanishing}
\end{document}